\documentclass[11pt]{amsart}
\usepackage{txfonts}
\usepackage{amsmath,amsthm}
\usepackage{amscd,amssymb}
\usepackage[all]{xy}
\usepackage{graphicx,calrsfs} 
\usepackage[colorlinks,plainpages,backref,urlcolor=blue,breaklinks=true]{hyperref}
\usepackage{breakurl}
\usepackage{bbm}
\usepackage{bbm}
\usepackage{mathrsfs}
\usepackage{xstring}
\usepackage{tikz}
\usetikzlibrary{matrix,arrows,decorations.pathmorphing,calc,shapes}

\allowdisplaybreaks[1]

 \newtheorem{theorem}{Theorem}[section]
\newtheorem{corollary}[theorem]{Corollary}
\newtheorem{lemma}[theorem]{Lemma}
\newtheorem{prop}[theorem]{Proposition}

\theoremstyle{definition}

\newtheorem{example}[theorem]{Example}
\newtheorem{remark}[theorem]{Remark}

\newtheorem*{ack}{Acknowledgments}

\DeclareMathOperator{\im}{im}
\DeclareMathOperator{\id}{id}

\DeclareMathOperator{\coker}{coker}

\DeclareMathOperator{\Ann}{Ann}
\DeclareMathOperator{\rank}{rank}

\DeclareMathOperator{\ab}{ab}

\DeclareMathOperator{\gr}{gr}
\DeclareMathOperator{\Sym}{Sym}
\DeclareMathOperator{\Hilb}{Hilb}

\DeclareMathOperator{\lcm}{LCM}
\DeclareMathOperator{\ini}{in}
\DeclareMathOperator{\ideal}{ideal}

\DeclareMathOperator{\Aut}{Aut}
\DeclareMathOperator{\Spec}{Spec}
\DeclareMathOperator{\spec}{Spec}
\DeclareMathOperator{\spn}{span}
\DeclareMathOperator{\Ass}{Ass}

\newcommand{\Q}{\mathbb{Q}}
\newcommand{\C}{\mathbb{C}}
\newcommand{\Z}{\mathbb{Z}}

\newcommand{\fB}{\mathfrak{B}}
\newcommand{\fh}{\mathfrak{h}}
\newcommand{\fg}{\mathfrak{g}}

\newcommand{\fm}{\mathfrak{m}}

\newcommand{\fI}{\mathfrak{I}}

\newcommand{\Lie}{\mathfrak{lie}}
\newcommand{\fq}{\mathfrak{q}}

\newcommand{\bL}{\mathbf{L}}
\newcommand{\V}{\mathbf{V}}

\newcommand{\bV}{\mathbf{V}}

\newcommand{\x}{\mathbf{x}}
\newcommand{\fS}{\mathfrak{S}}

\newcommand{\cD}{\mathcal{D}}

\newcommand{\cB}{\mathcal{B}}

\newcommand{\cE}{\mathcal{E}}

\newcommand{\cR}{\mathcal{R}}

\newcommand{\RR}{\mathcal{R}}
\newcommand{\cG}{\mathcal{G}}

\newcommand{\bg}{{\mathbf{g}}}
\newcommand{\ff}{{\mathbf{f}}}
\newcommand{\bh}{{\mathbf{h}}}

\def\dot{\mathchar"013A}  
\newcommand{\hdot}{{\raise1pt\hbox to0.35em{\huge $\dot$}}}

\newcommand{\bwedge}{\mbox{$\bigwedge$}}

\newcommand{\surj}{\twoheadrightarrow}
\newcommand{\inj}{\hookrightarrow}

\begin{document}


\title[Chen ranks and resonance varieties of the upper McCool groups]%
{Chen ranks and resonance varieties of the \\ upper McCool groups}

\author[Alexander~I.~Suciu]{Alexander~I.~Suciu$^1$}
\address{Department of Mathematics,
Northeastern University,
Boston, MA 02115, USA}
\email{\href{mailto:a.suciu@northeastern.edu}{a.suciu@northeastern.edu}}
\urladdr{\href{http://web.northeastern.edu/suciu/}%
{http://web.northeastern.edu/suciu/}}
\thanks{$^1$Supported in part by the Simons Foundation Collaboration 
Grant for Mathematicians \#354156}

\author{He Wang}
\address{Department of Mathematics and Statistics MS0084,
University of Nevada, Reno, NV 89557, USA}
\email{\href{mailto:wanghemath@gmail.com}%
{wanghemath@gmail.com}, \href{mailto:hew@unr.edu}%
{hew@unr.edu}}
\urladdr{\href{http://wolfweb.unr.edu/homepage/hew/}%
{http://wolfweb.unr.edu/homepage/hew/}}



\begin{abstract}
The group of basis-conjugating automorphisms of the free group 
of rank $n$, also known as the McCool group or the welded braid group 
$P\Sigma_n$, contains a much-studied subgroup, called the upper McCool 
group $P\Sigma_n^+$. Starting from the cohomology ring of 
$P\Sigma_n^+$, we find, by means of a Gr\"obner basis computation, 
a simple presentation for the infinitesimal Alexander invariant of this group, 
from which we determine the resonance varieties and the Chen ranks of 
the upper McCool groups.  These computations reveal that, 
unlike for the pure braid group $P_n$ and the full McCool 
group $P\Sigma_n$, the Chen ranks conjecture does not 
hold for $P\Sigma_n^+$, for any $n\ge 4$.  Consequently, 
$P\Sigma_n^+$ is not isomorphic to $P_n$ 
in that range, thus answering a question of Cohen, Pakianathan, 
Vershinin, and Wu. We also determine the scheme structure
of the resonance varieties $\cR_1(P\Sigma_n^+)$, and show 
that these schemes are not reduced for $n\geq 4$.
\end{abstract}

\maketitle
\setcounter{tocdepth}{1}
\tableofcontents

\section{Introduction}
\label{sec:Intro}

\subsection{Basis-conjugating groups}
\label{subsec:pure}

An automorphism of the free group $F_{n}=\langle x_1,\dots, x_n\rangle$ is 
called a symmetric automorphism if it sends each generator $x_i$ to a conjugate 
of $x_{\sigma(i)}$, for some permutation $\sigma\in \Sigma_n$. 
The set of all such automorphisms forms a subgroup $B\Sigma_n$ 
of $\Aut(F_n)$, known as the {\em braid-permutation group}, 
\cite{Cohen-P-V-Wu08} or the {\em welded braid group}, 
\cite{Bar-Natan-Dancso16}.  The Artin braid group $B_n$ is 
the subgroup of $B\Sigma_n$  consisting of those symmetric 
automorphisms which fix the word $x_1\cdots x_n$. 
 
The kernel of the canonical projection $B\Sigma_n\surj S_n$, denoted $P\Sigma_n$, 
is known as the {\em basis-conjugating group}, or the {\em pure welded braid group}.  
In \cite{McCool86}, J.~McCool  showed that  $P\Sigma_n$  
is generated by the Magnus automorphisms $\alpha_{ij}\colon x_i \mapsto 
x_jx_ix_j^{-1}$, for all $1\le i\ne j\le n$, and gave a presentation of this group; 
for that reason, $P\Sigma_n$ is also known as the {\em McCool group}.  
Notably, the group $P\Sigma_n$ can be realized as the pure motion 
group of $n$ unknotted, unlinked circles in $S^3$.  We refer 
to the recent surveys \cite{Damiani17, SW-pisa} for detailed accounts 
of this subject and further references.

We concentrate in this paper on the subgroup of $P\Sigma_n$ 
generated by the automorphisms $\alpha_{ij}$ with $i>j$.  This subgroup  
is called the {\em upper triangular McCool group}, and is denoted by $P\Sigma_n^{+}$.
Both the pure braid group $P_n=\ker(B_n\surj S_n)$ and the upper 
McCool group  $P\Sigma_n^+$ 
are subgroups of the full McCool group $P\Sigma_n$. 
Furthermore, both groups are iterated semidirect products 
of the form $F_{n-1}\rtimes \dots \rtimes F_2\rtimes F_1$, with 
monodromies acting trivially in first homology; thus, they 
share the same Betti numbers and the same lower central series 
quotients, see \cite{Arnold69,  FalkRandell, Kohno85, Cohen-P-V-Wu08}.

In \cite{Cohen-P-V-Wu08}, Cohen, Pakianathan, Vershinin, and Wu  
asked whether or not the groups $P_n$ and  $P\Sigma_n^+$ 
are isomorphic. For $n\leq 3$, it was already known that the answer is yes.  
In \cite{Bardakov-Mikhailov08}, Bardakov and Mikhailov attempted to 
prove that $P_4$ is not isomorphic to $P\Sigma_4^+$ by showing that
the two groups have different single-variable Alexander polynomials. 
However, the single-variable Alexander polynomial depends on the choice 
of presentation for a group, and thus it cannot be 
used as an isomorphism-type invariant.  Moreover, the multi-variable Alexander 
polynomial (which is an isomorphism-type invariant), is equal to $1$ for 
both $P_4$ and $P\Sigma_4^+$.  

Nevertheless, the work that we undertake here  allows us to distinguish 
the groups $P_n$ and  $P\Sigma_n^+$ for all $n\ge 4$, by means of 
both the Chen ranks and the resonance varieties associated to these groups.  
Some of these results were announced in \cite{SW-pisa}; this 
paper contains full proofs of those  results.  

\subsection{Chen ranks}
\label{subsec:intro-chen}

Given a finitely generated group $G$, we let $\{\Gamma_k G\}_{k\ge 1}$ 
be its lower central series, and we let 
$\gr(G)=\bigoplus_{k\ge 1} \Gamma_k G/\Gamma_{k+1} G$ 
be the associated graded Lie ring, with Lie bracket induced from the group commutator.   
The LCS ranks of $G$, then, are the integers $\phi_k(G)=\rank \gr_k(G)$. 

The Chen ranks of $G$, introduced by K.-T.~Chen in \cite{Chen51}, are the LCS 
ranks of the quotient of $G$ by its second derived subgroup, $G''$:
\[
\theta_k(G):=\rank \gr_k(G/G'').
\]
As such, the Chen ranks provide an approximation from below for the LCS ranks. 
For a variety of reasons, though, the Chen ranks $\theta_k(G)$ are invariants worth 
studying in their own right, oftentimes providing more refined information about the 
given group $G$ than the LCS ranks $\phi_k(G)$. 

In \cite{Massey80}, W.~Massey used the Chen ranks to study the fundamental groups of 
link complements.  In the process, he showed that the Chen ranks of a finitely generated 
group $G$ can be computed from the Alexander invariant 
$B(G):=H_1(G^{\prime};\C)$, which is the first homology of the commutator subgroup 
$G^{\prime}=\Gamma_2 G$, viewed as a module over the group algebra 
$R=\C[H_1(G;\Z)]$, as follows:
\[
\theta_k(G)=\dim_{\C} \gr_{k-2} (B(G)), \textrm{ for $k\ge 2$},
\]
where $\gr (B(G))$  is the associated graded module to $B(G)$ 
with respect to the filtration by powers of the augmentation ideal 
$I=\ker (\varepsilon \colon R\to \C)$.

A quadratic approximation of the Lie algebra $\gr(G)\otimes \C$ is the 
\emph{holonomy Lie algebra}\/ of $G$ defined by $\fh(G):=\Lie(H_1(G;\C))/\im(\partial_G)$,
where $\Lie(H_1(G;\C))$ the free Lie algebra generated by the first homology $H_1(G;\C)$, 
and $\partial_G$ is the dual of the cup product map 
$H^1(G;\C)\wedge H^1(G;\C)\to H^2(G;\C)$.
The infinitesimal Alexander invariant of $G$ is the finitely generated, 
graded $S$-module defined by $\fB(G):=\fh(G)^{\prime}/\fh(G)^{\prime\prime}$,
where $S=\gr(R)$ is the symmetric algebra on $H_1(G;\C)$.   
If the group $G$ is $1$-formal (in the sense of rational 
homotopy theory), then, as shown by Papadima and Suciu in \cite{Papadima-Suciu04}, 
there is an isomorphism of graded $S$-modules, $\gr(B(G))\cong \fB(G)$. 
Thus, the Chen ranks of such groups $G$ can 
be computed from the Hilbert series of $\fB(G)$.

The class of $1$-formal groups to which the above method applies includes all 
arrangement groups (such as the pure braid groups $P_n$), K\"ahler groups, 
and right-angled Artin groups, see for instance 
\cite{Dimca-Papadima-Suciu09, Papadima-Suciu09, SW-formality} 
and references therein.  Of great importance to us is that, as shown by 
Berceanu and Papadima in \cite{Berceanu-Papadima09}, all the McCool groups 
$P\Sigma_n$ and $P\Sigma^+_n$ are $1$-formal.

Based on a refinement of the Gr\"obner basis algorithm from \cite{Cohen-Suciu95} 
applied to the infinitesimal Alexander invariant $\fB(P\Sigma^+_n)$, we find a closed 
formula for the Chen ranks of the groups $P\Sigma_n^+$.

\begin{theorem}[Theorem \ref{thm:chenranks}]
\label{thm:ChenRanksIntro}
The Chen ranks of the upper McCool groups, $\theta_k=\theta_k(P\Sigma_n^+)$, 
are given by $\theta_1=\binom{n}{2}$, $\theta_2=\binom{n}{3}$, and
\begin{equation*}
\theta_k= \binom{n+1}{4} + \sum_{i=3}^k\binom{n+i-2}{i+1}  \quad \text{for $k\geq 3$}.
\end{equation*}
\end{theorem} 
 
As a quick application of our result, we obtain the following corollary, which answers 
the aforementioned question of F.~Cohen et al. from \cite{Cohen-P-V-Wu08}.

\begin{corollary}[Corollary \ref{cor:chen ppp}]
\label{cor:IntroIso}
For each $n\ge 4$, the pure braid group $P_n$, the upper 
McCool group $P\Sigma_n^+$, and the direct product 
$\Pi_n:= \prod_{i=1}^{n-1}F_{i}$ are all pairwise non-isomorphic, 
although they all do have the same LCS ranks and the 
same Betti numbers.
\end{corollary}

The fact that $P_n \not\cong \Pi_n$ for $n\ge 4$ was already 
established by Cohen and Suciu in \cite{Cohen-Suciu95}, 
also using the Chen ranks.  The 
novelty here is the distinction between $P\Sigma_n^+$ 
and the other two groups. 

\subsection{Resonance varieties}
\label{subsec:intro-res}

Given a finitely generated group $G$, we let $A^*=H^*(G;\C)$ be its 
cohomology algebra. The resonance varieties of $G$ are the jump loci 
for the cohomology of the Aomoto complexes $(A,a)$ parametrized by 
the vector space $A^1$. 
We focus here on the \emph{first resonance variety}, $\cR_1(G)$, which is 
defined as  
\begin{equation}
\label{eq:r1g}
\cR_1(G)=\big\{a\in A^1\mid \text{$\exists \, b\in A^1$ such that 
$b\notin \C\cdot a$ and $ab=0\in A^2$}\big\}.
\end{equation}

In general, these varieties can be arbitrarily complicated  
homogeneous algebraic subsets of $A^1$.  Nevertheless, 
if the group $G$ is $1$-formal, then the Tangent Cone theorem of 
\cite{Dimca-Papadima-Suciu09} insures that $\cR_1(G)$ 
is a union of rationally defined linear subspaces of $H^1(G;\C)$. 
For instance, the first resonance variety of the pure braid group $P_n$,
determined in \cite{Cohen-Suciu99}, is a union of linear subspaces of $H^1(P_n;\C)$
of dimension $2$.  In \cite{DCohen09}, D.~Cohen computed the first resonance variety 
of the full McCool group $P\Sigma_n$, showing that this variety is a union of linear 
subspaces of $H^1(P\Sigma_n;\C)$ of dimension $2$ and $3$.

In this paper, we pursue this line of inquiry by determining the 
resonance varieties of the upper McCool groups $P\Sigma_n^+$.  
To start with, let us identify the ambient space $H^1(P\Sigma_n^+;\C)$ 
with $\C^{\binom{n}{2}}$, and pick coordinate functions $x_{i,j}$ with $1\le j<i\le n$ 
corresponding to the Magnus generators $\alpha_{ij}$.  By \cite{Matei-Suciu00},  
the resonance variety $\cR_1(P\Sigma_n^+)$ is cut out by the annihilator ideal 
of the infinitesimal Alexander invariant $\fB_n=\fB(P\Sigma^+_n)$, viewed as a 
module over the coordinate ring $S=\C[x_{ij}]$.  Using the aforementioned 
Gr\"obner basis for $\fB_n$, we arrive at the following description of the 
variety $\cR_1(P\Sigma_n^+)$.

\begin{theorem}[Theorems \ref{thm:Resonance} and \ref{thm:resonance}]
\label{thm:ResonanceIntro}
For each $n\ge 3$, the first resonance variety of the upper McCool group $P\Sigma_n^+$ 
decomposes into irreducible components as 
\begin{equation*}
\cR_1(P\Sigma_n^+)=\bigcup\limits_{2\leq j<i\leq n} L_{ij},
\end{equation*}
where $L_{ij}$ is the $j$-dimensional linear subspace 
of $H^1(P\Sigma_n^+;\C)$ defined by the equations 
\begin{equation}
\label{eq:intro L_ij}
\left\{
\begin{aligned}
\notag
x_{i,l}+x_{j,l}&=0 &&\text{for $1\leq l\leq j-1$},\\
x_{i,l}&=0 &&\text{for  $j+1\leq l\leq i-1$},\\
x_{s,t}&=0 &&\text{for $s\neq i$, $s\neq j$, and $1\leq t<s$}.
\end{aligned}
\right.
\end{equation} 
Moreover, 
\begin{enumerate}
\item 
All the components are projectively disjoint, i.e.,  
$L_{ij}\cap L_{st}=\{0\}$ if $(i,j)\neq (s,t)$.
\item 
The subspaces $L_{ij}$ are $0$-isotropic for $j=2$ 
and $\binom{j-1}{2}$-isotropic for $j\geq 3$, i.e.,  
the restriction of the cup-product map on $H^1(P\Sigma_n^+;\C)$ 
to $L_{ij}$ has rank $\binom{j-1}{2}$.
\end{enumerate}
\end{theorem}

A finitely presented group $G$ is said to be \emph{quasi-projective}\/ if it can 
be realized as the fundamental group of a smooth, complex, 
quasi-projective variety. A classical problem, formulated by J.-P. Serre, 
is to determine which finitely presented groups are quasi-projective. 
Using the aforementioned description of the first resonance varieties of  
$P\Sigma_n^+$ and structural theorems from 
\cite{Dimca-Papadima-Suciu08, Dimca-Papadima-Suciu09}, 
we obtain the following result.

\begin{prop}[Proposition \ref{prop:quasiproj}]
\label{prop:intro quasiproj}
The upper McCool groups $P\Sigma_n^+$ are not quasi-projective groups, 
for any $n\geq 4$.  
\end{prop}

Comparing the resonance varieties of $P\Sigma_n^+$ with those of $P_n$ 
and $\Pi_n$ (already computed in \cite{Cohen-Suciu99}), 
we obtain another proof of Corollary \ref{cor:IntroIso}.
Furthermore, comparing the resonance varieties of $P\Sigma_n$ with those 
of $P\Sigma_n^+$, we obtain the following application.

\begin{prop}[Proposition \ref{prop:split}]
\label{prop:intro epi}
There is no epimorphism from $P\Sigma_n$ to $P\Sigma_n^+$ for $n\geq 4$. 

\end{prop}

\subsection{Resonance scheme structure}
\label{subsec:introScheme}
As shown by Matei and Suciu in \cite{Matei-Suciu00}, 
the resonance variety $\cR_1(G)$ of a commutator-relators 
group $G$ coincides with the support variety of the annihilator of $\fB(G)$. 
It is natural then to talk about the \emph{resonance scheme}\/ of $G$
as the scheme defined by the ideal $\Ann(\fB(G))$.   The primary 
components of this ideal cut out the primary subschemes; the resonance 
scheme consists of isolated components (namely, the 
irreducible components of $\cR_1(G)$), together with 
embedded components. We say that $\cR_1(G)$ is 
\emph{weakly reduced}\/ as a scheme if the only embedded 
component of $\Ann (\fB(G))$ is the point $0\in H^1(G;\C)$.

The next theorem describes the 
resonance scheme structure of the upper McCool groups.

\begin{theorem}[Theorem \ref{thm:resonance scheme}]
\label{thm: intro scheme}
The resonance scheme of $P\Sigma_n^+$ consists of: 
\begin{description}
\item[Isolated components] The linear subspaces $L_{ij}$  ($2\le j<i\le n$)
listed in Theorem \ref{thm:ResonanceIntro}.
\item[Embedded components]  The $1$-dimensional 
linear subspaces $L_{ij}' \subset L_{ij}$ ($3\le j<i\le n$) 
defined by the equations $x_{st}=0$ with $1\leq t<s\leq n$ 
and $(s,t)\neq (i,j)$.
\end{description}
\end{theorem}

In particular (Corollary \ref{cor:notreduced}), the resonance variety 
$\cR_1(P\Sigma_n^+)$ is not weakly reduced as a scheme for $n\geq 4$.

The scheme structure of $\cR_1(G)$ is crucial for studying the 
relationship between the first resonance variety and the Chen ranks 
of a group $G$. It was conjectured in \cite{Suciu01} that the 
Chen ranks of an arrangement group $G$ are given by
\begin{equation}
\label{eq:ChenranksConj}
\theta_k(G)= \sum_{m\geq 2}h_m(G) \cdot \theta_k(F_m),   \textrm{ for $k\gg 0$},
\end{equation}
where $h_m(G)$ is the number of $m$-dimensional components of $\cR_1(G)$.
In \cite{Cohen-Schenck15}, Cohen and Schenck proved the conjecture for 
the wider class of $1$-formal groups which have weakly reduced resonance 
schemes, and $0$-isotropic, projectively disjoint irreducible components.
They also showed that the first resonance variety of the full McCool 
group $P\Sigma_n$ is weakly reduced.   Furthermore,  the Chen ranks 
formula works for $P\Sigma_n$, from which they deduced that 
$\theta_k(P\Sigma_n) = (k-1)\binom{n}{2} + (k^2-1) \binom{n}{3}$  for $k\gg 0$. 

Theorem \ref{thm:ChenRanksIntro} provides a closed formula for the Chen ranks 
$\theta_k(P\Sigma_n^+)$ for $k\geq 1$. Rather surprisingly, it turns out that 
the Chen ranks formula does not apply to $P\Sigma_n^+$ for $n\geq 4$. 
There are two reasons for that: firstly, the resonance variety $\cR_1(P\Sigma_n^+)$ 
contains non-isotropic components, and secondly, $\cR_1(P\Sigma_n^+)$ is not 
weakly reduced as a scheme.  The computation of the scheme structure also 
shows how the embedded components affect the Chen ranks. This provides 
us with a benchmark test case for a generalized Chen ranks formula, which  
is the subject of ongoing work.
 
\section{Alexander invariant and Chen ranks}
\label{sect:alexchen}

We start by reviewing several  invariants 
associated to a finitely generated group, mainly the 
lower central series ranks, the Chen ranks, and the 
Alexander invariant.  We then discuss the infinitesimal version of the Alexander 
invariant, this time associated to a graded Lie algebra.

\subsection{Associated graded Lie ring and Chen ranks}
\label{subsec:chen}
Throughout, $G$ will be a finitely generated group. 
The terms of the lower central series (LCS) of $G$ 
are defined inductively by $\Gamma_1 G=G$ and 
$\Gamma_kG=[G,\Gamma_{k-1}G]$ for $k\geq 2$.  It is readily seen 
that $\Gamma_{k+1}G$ is a normal subgroup of $\Gamma_{k}G$, 
and the quotient group, $\gr_k(G)=\Gamma_{k}G/\Gamma_{k+1}G$, 
is a finitely generated abelian group. 
The {\em associated graded Lie ring}\/ of $G$ is the direct sum
\begin{equation}
\label{eq:gig}
\gr(G)= \bigoplus_{k\geq 1} \Gamma_{k}G/\Gamma_{k+1}G,
\end{equation}
with Lie bracket $[\:,\:]\colon \gr_k(G)\times \gr_{\ell}(G)\to \gr_{k+\ell}(G)$ induced 
by the group commutator.   By definition, the LCS ranks of $G$ are the integers 
$\phi_k(G):=\rank \gr_k(G)$.

Now let $G'=\Gamma_2 G$ be the derived subgroup of $G$, 
and let $G'' =[G',G']$ the second derived subgroup.  
Then $G_{\ab}:=G/G^{\prime}$ is the maximal abelian quotient of $G$, 
whereas $G/G''$ is the maximal metabelian quotient of $G$. 
Following  \cite{Chen51, Chen73}, let us define the {\em Chen ranks}\/ 
of $G$ as the LCS ranks of $G/G''$:
\begin{equation}
\label{eq:chen ranks}
\theta_k(G):=\rank \gr_k(G/G^{\prime\prime}).
\end{equation}
It is readily seen that $\theta_k(G)\le \phi_k(G)$, with equality for $k\le 3$. 

\subsection{Alexander invariant and Chen ranks}
\label{subsec:AlexInv}

Let $\Z[G]$ be the group ring of $G$, let  $\varepsilon\colon \Z{G}\to \Z$ 
be the augmentation homomorphism, defined by $\varepsilon(g)=1$ for 
$g\in G$, and let $I=\ker \varepsilon$ be the augmentation ideal. 
The \emph{Alexander module}, $A(G)= I \otimes_{\Z{G}} \Z{G_{\ab}}$, 
is the $\Z{G_{\ab}}$-module induced from $I$ by the extension of 
the abelianization map $\alpha\colon G \rightarrow G/G^{\prime}$ to group rings.
The \emph{Alexander invariant}\/ of $G$ is the $\Z{G}_{\ab}$-module 
\[
B(G)=G^{\prime}/G^{\prime\prime},
\]
with the group $G_{\ab}$ acting on the cosets of $G^{\prime\prime}$ 
via conjugation. Since the group $G$ is finitely generated, both   
$A(G)$ and $B(G)$ are finitely generated $\Z{G}_{\ab}$-modules.
 
As shown by W. Massey in \cite{Massey80}, the Chen ranks can be computed 
from the Alexander invariant using the group extension
\begin{equation}
\label{eq:bgmod}
\xymatrixcolsep{18pt}
\xymatrix{
0\ar[r] &G'/G''\ar[r] &G/G'' \ar[r] & G/G'\ar[r] & 0
}.
\end{equation}

More precisely, let us filter both the group ring $\Z{G_{\ab}}$ 
and the module $B(G)$ by the powers of the 
augmentation ideal $J=\ker(\varepsilon_{\ab}\colon \Z{G_{\ab}}\to \Z)$, 
and let us take the associated graded module, 
$\gr(B(G))=\bigoplus_{k\ge 0} J^kG/J^{k+1}G$, 
viewed as a module over the ring $\gr(\Z{G_{\ab}})$.  Identifying this ring 
with the symmetric algebra $S=\Sym(G_{\ab})$ in a canonical fashion,   
we may view $\gr(B(G))$ as a (finitely generated) $S$-module.  
The following equality then holds, 
for all $k\geq 0$:
\begin{equation}
\label{eq:massey}
\theta_{k+2}(G)=\rank \gr_{k}(B(G))\, .
\end{equation}


 
\subsection{Infinitesimal Alexander invariant of an algebra}
\label{subsec:infiAlexInv}
 
Let $\fg=\bigoplus_{k\ge 1} \fg_k$ be a finitely generated, graded 
Lie algebra over $\C$.  We denote by $S$ the universal enveloping 
algebra of its abelianization, $\fg/\fg^{\prime}$.  We will identify this 
algebra with the symmetric algebra $S=\Sym(\fg_1)$, with variables in 
degree $1$. 

Following \cite{Papadima-Suciu04}, let us define 
the \textit{infinitesimal Alexander invariant}\/ of $\fg$ 
to be the graded $S$-module 
\begin{equation}
\label{eq:infalex}
\fB(\fg):=\fg^{\prime}/\fg^{\prime\prime}.
\end{equation}
The exact sequence of graded Lie algebras
\begin{equation}
\label{eq:exacts}
\xymatrixcolsep{18pt}
\xymatrix{
0\ar[r] &\fg^{\prime}/\fg^{\prime\prime}\ar[r] &
\fg/\fg^{\prime\prime}\ar[r] & \fg/\fg^{\prime}\ar[r] & 0
}
\end{equation}
defines the required graded $S$-module structure on $\fB(\fg)$.
 
Now let $A=\bigoplus_{i\geq 0}A^i$ be graded, graded-commutative 
algebra over $\C$. We shall assume that $A$ is connected (i.e., 
$A^0=\C$, generated by the unit $1$), and locally finite 
(i.e., $A^i$ has finite dimension, for each $i\geq 1$). 
Write $V=A^1$, and let $\partial_A\colon (A^2)^*\rightarrow V^*\wedge V^*$ 
be the dual of the multiplication map $\mu_A \colon V\wedge V\to A^2$, 
where we identified $(V\wedge V)^* \cong V^*\wedge V^*$. 
The \textit{holonomy Lie algebra}\/ of $A$ is defined to be the quotient 
\begin{equation}
\label{eq:hololiealg}
\fh(A)=\Lie(V^*)/\langle\im \partial_A\rangle 
\end{equation}
of the free Lie algebra on $V^*$ by the ideal generated by the 
image of $\partial_A$. By construction, $\fh(A)$ is a finitely 
presented, quadratic Lie algebra.  

By definition, the \emph{infinitesimal Alexander invariant}\/ 
of $A$ is the graded $S$-module 
\begin{equation}
\label{eq:fba}
\fB(A):=\fB(\fh(A)),
\end{equation}
where $S=\Sym(\fh_1(A))$ is canonically identified with $\Sym(V^*)$.
From the exact sequence \eqref{eq:exacts}, we have the equality
\begin{equation}
\label{eq:holoChenHilbert}
\sum\limits_{k\geq0}\theta_{k+2}(A)\cdot t^k={\rm Hilb}(\mathfrak{B}(A),t),
\end{equation}
where $\theta_{i}(A):=\dim((\fh(A)/\fh(A)'')_i)$ is the $i$-th \emph{Chen rank}\/ 
of $A$. It is readily seen  that 
$\fh(A)$ and $\fB(A)$ coincide with the holonomy Lie algebra and the
infinitesimal Alexander invariant of the 
quadratic closure of $A$, 
\begin{equation}
\label{eq:quadclose}
\bar{A}=E /\langle \ker\mu_A\rangle, 
\end{equation} 
where $E=\bigwedge V$  and $\langle \ker\mu_A\rangle$ is the 
ideal generated by $\ker \mu_A$. We refer to \cite{SW-mz, SW-formality, SW-holo} 
for full details of this construction, and further references and background.

\subsection{Infinitesimal Alexander invariant of a $1$-formal group}
\label{subsec:chen-formal}

Let $G$ be a finitely generated group, and   
suppose that $H^{2}(G;\C)$ is finite-dimensional. 
We define then the \emph{holonomy Lie algebra}\/  
of $G$ as $\fh(G):=\fh(A)$, where $A=H^{\le 2}(G,\C)$ 
is the degree-$2$ truncation of the cohomology algebra of $G$.  
The \emph{infinitesimal Alexander invariant}\/ of $G$ is 
then the graded $S$-module $\fB(G):=\fB(A)$,  where 
$S$ is the symmetric algebra on $H_1(G;\C)$.  
Note that this module is finitely generated (in degree $0$), 
and therefore admits a finite presentation of the form
$S^m\xrightarrow{\alpha} S^n \rightarrow \fB(G)$.
Choosing bases for $S^m$ and $S^n$, we may view the map $\alpha$
as a matrix with $n$ rows and $m$ columns, having entries in $S$.  

For each nilpotent quotient $G/\Gamma_iG$, there is a filtered 
$\C$-Lie algebra $\fm(G/\Gamma_iG)$, whose construction goes 
back to Anatoli Malcev. 
The \emph{Malcev Lie algebra}\, of $G$ is defined to be the inverse limit 
$\fm(G):=\varprojlim_k \fm(G/\Gamma_k G)$, see for instance 
\cite{Papadima-Suciu04, Papadima-Suciu09, SW-formality} for details and references.
We say that the group $G$ is \emph{$1$-formal}\/ if there exists a filtered Lie algebra isomorphism 
between the Malcev Lie algebra $\fm(G)$ and the degree completion of $\fh(G)$.

If $G$ is a $1$-formal group, then, as shown in \cite{Papadima-Suciu04} in the 
commutator-relators case, and in \cite{SW-holo} in general, the following equality holds:
\begin{equation}
\label{eq:ChenHilbert}
\sum\limits_{k\geq0}\theta_{k+2}(G)\cdot t^k={\rm Hilb}(\mathfrak{B}(G),t).
\end{equation}

\subsection{Presentations for $\fB(A)$}
\label{subsec:presba}

Now suppose $\fg$ admits a finite, quadratic presentation, that is,
$\fg=\Lie(H)/\langle K \rangle$, where $H$ is a finite-dimensional 
$\C$-vector space, $K$ is a finite set of degree-two elements in 
the free Lie algebra $\Lie(H)$, and $\langle K \rangle$ is the 
Lie ideal generated by $K$.  Then, by \cite{Papadima-Suciu04}, 
the $S$-module $\fB(\fg)$ admits a homogeneous, finite presentation 
of the form 
\begin{equation}
\label{eq:PresentationPS04}
\xymatrixcolsep{18pt}
\xymatrix{ \Big(\big(\bigwedge^3H\big)  \oplus K \Big) \otimes S
\ar^(.58){\delta_3+(\id\otimes\, \iota)}[rr]&& 
 \bigwedge^2H \otimes S \ar[r]& \fB(\fg)  \ar[r]& 0},
\end{equation}
where $\iota$ is the inclusion of $\mathfrak{a}$ into $\Lie(H)_2\cong H\wedge H$, 
and $\delta_3$ is the Koszul differential.

As before, let $A$ be a graded, graded-commutative, locally finite, 
connected algebra. The algebra $A$ may be viewed as an $E$-module, 
where $E=\bigwedge V$ is the exterior algebra on the vector space 
$V=A^1$.  Pick a basis $\{e_1,\dots, e_n\}$ for $V$ and let 
$\{x_1,\dots,x_n\}$ be the dual basis for $V^*$; 
then $S=\Sym(V^*)$ may be identified with the polynomial ring $\C[x_1,\dots, x_n]$.
When applied to the $E$-module $A$, the Bernstein--Gelfand--Gelfand correspondence 
(see e.g.~\cite[\S 7B]{Eisenbud05}) yields a cochain complex of free $S$-modules, 
\begin{equation}
\label{eq:universalAomoto}
\xymatrix{
\bL(A): \  \: A^0\otimes S\ar[r]^(.65){d^0}&A^1\otimes S\ar[r]^{d^1}&A^2\otimes S
\ar[r]^(.53){d^2} &\cdots, 
}
\end{equation}
with differentials given by 
\begin{equation}
\label{eq:AomotoDiff}
d^i(u\otimes s)=\sum\limits_{j=1}^ne_ju\otimes x_js
\end{equation}
for $u\in A^i$ and $s\in S$. In particular, $\bL(E)$ is the dual of the Koszul complex.

Let $I:=\langle \ker\mu_A\rangle$ be the (graded) ideal of $E$ generated by 
$\ker \mu_A$ as in \eqref{eq:quadclose},
and denote $\iota\colon I\to E$ the inclusion map.
 By construction, $I^2=\ker \mu_A$.  
Hence, we have a commuting diagram,
\begin{equation}
\label{diagram:Phi}
\begin{gathered}
\xymatrixcolsep{28pt}
\xymatrix{
0\ar[r] &I^2\otimes S\ar[rd]^{\Phi} \ar[r]^{\iota\otimes \id}\ar[d]_{d^2_I} 
&E^2\otimes S\ar[d]_{d^2_E} \ar[r]^{\mu_A\otimes \id} 
 &\bar{A}^2\otimes S\ar[d]_{d^2_{\bar{A}}} \ar[r]&0 \phantom{\, .}\\
0\ar[r] &  I^3\otimes S\ar[r] &E^3\otimes S \ar[r] 
&\bar{A}^3\otimes S \ar[r]&0\, ,
}
\end{gathered}
\end{equation}
where $\Phi$ is the composite  
$d^2_E \circ (\iota\otimes \id) \colon I^2\otimes S\rightarrow E^3\otimes S$.  
  
In the next lemma, we obtain another presentation 
for $\fB(A)$ which has a minimal generating set. 
This result generalizes formula (2.5) from \cite{Cohen-Schenck15}, where 
Cohen and Schenck give a presentation for the linearized Alexander invariant 
of a commutator-relators group.

\begin{lemma}
\label{lem:presentationAlex}
The dual of the $S$-linear map $\Phi\colon I^2\otimes S\rightarrow E^3\otimes S$ 
provides a presentation for the infinitesimal Alexander invariant $\fB(A)$.
In addition, any basis for the vector space dual of $I^2$ gives a minimal generating 
set for the $S$-module $\fB(A)$. 
 \end{lemma}
 
\begin{proof}
By definition, the ideal $I$ of the exterior algebra $E=\bigwedge V$ 
is generated by the vector space $I^2=\ker \mu_A$.  Taking the dual spaces,  
we have an isomorphism $(I^2)^*\cong \coker \partial_A$, where $\partial_A$ is the dual of
the multiplication $\mu_A$.
Recall that the map $\Phi$ was defined as the composite 
\begin{equation}
\label{eq:PhiDiagram}
\xymatrixcolsep{28pt}
\xymatrix{
I^2\otimes S \ar@/^1.5pc/[rr]^{\Phi} \ar[r]^(.46){\iota \otimes \id} 
&\bigwedge^2 V\otimes S \ar[r]^(.43){d^2_E} &\bigwedge^3 V\otimes S
},
\end{equation}
 where $d^2_E$ is the dual of the Koszul differential.
 All $S$-modules in the above diagram are free modules. 
Taking duals, we obtain the diagram 
\vspace{-2pt}
\begin{equation}
\label{eq:PhiDiagramDual}
\xymatrix{
\coker(\partial_A)\otimes S&\bigwedge^2 V^*\otimes S  \ar[l]_(.4){(\iota\otimes\id)^*}
&\bigwedge^3 V^*\otimes S\ar@/_1.6pc/[ll]_{\Phi^*}  \ar[l]_(.53){(d_E^2)^*} 
}.
\end{equation}
Here, $(\iota\otimes \id)^*$ coincides with the projection map and $(d^2_E)^*$ coincides with the Koszul differential 
$\delta_3$ from \eqref{eq:PresentationPS04} in the case when $\fg=\fh(A)$.
It follows that $\fB(A)$ is isomorphic to $\coker \Phi^*$, as claimed.

To prove the last assertion,  we may assume without loss of generality 
that $A=\bar{A}=E/\langle I^2\rangle $. 
From \eqref{eq:holoChenHilbert}, we have that 
$\dim \fB(A)_0 = \theta_2(A)=\dim \fh(A)_2$;  
in view of \eqref{eq:hololiealg}, then, 
\begin{equation}
\label{eq:dimb0}
\dim \fB(A)_0  = \dim\big(\bwedge^2 V^*\big)-\dim (\im \partial_A)=
\dim (\coker \partial_A)=\dim I^2.
\end{equation}
Since the $S$-module  $\fB(A)$ is generated by $\fB(A)_0$, we 
conclude that the generating set for $\fB(A)$ given by a basis 
for the dual of $I^2$ is indeed a minimal generating set.
\end{proof} 

\section{The infinitesimal Alexander invariant of the upper McCool groups}
\label{sec:inf alexinv}
In this section, we give a presentation for the infinitesimal Alexander invariant 
of the upper McCool groups, and simplify this presentation to a minimal presentation.
 
\subsection{The upper McCool groups}
\label{subsec:upperMcCool} 
Let $F_n$ be the free group on generators $x_1,\dots, x_n$, and let 
$\Aut(F_n)$ be its automorphism group.  Recall that the  basis-con\-jugating group 
$P\Sigma_n$ is the subgroup of $\Aut(F_n)$ consisting of those automorphisms 
which send each generator $x_i$ to a conjugate of itself.  

In \cite{McCool86}, James McCool gave a presentation for  
$P\Sigma_n$, a group also known nowadays as the {\em McCool group}, or 
the pure welded braid group.   This presentation has generators $\alpha_{ij}$ 
(the automorphism sending $x_i$ to $x_jx_ix_j^{-1}$) 
for $1\leq i\neq j\leq n$, and relations
\begin{align}
\label{eq:McCoolGroup}
\alpha_{ij}\alpha_{ik}\alpha_{jk}&=\alpha_{jk}\alpha_{ik}\alpha_{ij} \notag\\
[\alpha_{ik},\alpha_{jk}]&=1 &&\text{for distinct $i,j,k$},\\
[\alpha_{ij},\alpha_{st}]&=1 &&\text{if $\{i,j\}\cap \{s,t\}=\emptyset$}.\notag
\end{align}
It follows at once that $P\Sigma_1=\{1\}$ and $P\Sigma_2=F_2$.   Since 
there is not much else to be said in those cases, we will usually concentrate 
on the case when $n\ge 3$.

The subgroup of $P\Sigma_n$ generated by the elements $\alpha_{ij}$ with 
$i>j$ is called the \textit{upper triangular McCool group}, and is denoted by 
$P\Sigma_n^{+}$.  It readily seen that $P\Sigma_1^+=\{1\}$, $P\Sigma_2^+= \Z$,  
and $P\Sigma_3^+\cong F_2\times \Z$.

Work of Berceanu and Papadima from \cite{Berceanu-Papadima09} 
establishes the $1$-formality of all these groups.

\begin{theorem}[\cite{Berceanu-Papadima09}]
\label{thm:McCoolFormal}
The McCool groups $P\Sigma_n$, as well as their upper triangular 
subgroups $P\Sigma_n^{+}$ are $1$-formal, for all $n\ge 1$. 
\end{theorem}

In \cite{Jensen-McCammond-Meier06}, Jensen, McCammond, 
and Meier computed the cohomology ring of $P\Sigma_n$, 
thereby verifying a long-standing conjecture of Brownstein and Lee.  
Shortly after, the integral cohomology ring of $P\Sigma_n^+$ was 
computed by F.~Cohen et al.~\cite{Cohen-P-V-Wu08}, as follows.

\begin{theorem}[\cite{Cohen-P-V-Wu08}]
\label{thm:Cohen-P-V-Wu}
The cohomology algebra $A=H^*(P\Sigma_n^+;\Z)$ is the graded, graded-commutative
(associative) algebra generated by degree $1$ elements $u_{ij}$ with $1\leq j<i\leq n$,
subject to the relations $u_{ij}(u_{ik}-u_{jk})=0$ for $k<j<i$.
\end{theorem}
 
We will use Theorem \ref{thm:Cohen-P-V-Wu} to compute a presentation 
for the infinitesimal Alexander invariant  $\fB_n:=\fB(P\Sigma_n^+)$.
We first choose an order for the aforementioned basis of $H^1(P\Sigma_n^+;\C)$ 
by setting $u_{ij}{\succ}u_{kl}$ if either $i>k$, or $i=k$ and $j>l$.

Let $\x=\{x_{ij}\mid 1\leq j<i\leq n\}$ be the dual of the basis $\{u_{ij}\}$ 
of $H^1(P\Sigma_n^+;\Z)$, and let $S=\C[\x]$ be the polynomial ring 
in those variables. The relators
\begin{equation}
\label{eq:basisI2}
\{ r_{ijk}^*:=(u_{ik}-u_{jk})u_{ij} \mid 1\leq k<j<i\leq n\}
\end{equation}
for the cohomology algebra $A=E/I$ from Theorem \ref{thm:Cohen-P-V-Wu} 
form a basis for the vector space $I^2$, as well as for the free $S$-module 
$I^2\otimes S$.  Finally, the set
$\left\{u_{st}u_{lk} u_{ij}\mid u_{ij}{\succ}u_{lk}{\succ}u_{st} \right\}$ 
forms a basis for $E^3$, and also a basis for the free $S$-module $E^3\otimes S$.
 
\subsection{The map $\Phi$}
\label{subsec:presinf}

Using the aforementioned choices of bases, we now 
provide an explicit description of the map  $\Phi$ 
from diagram \eqref{diagram:Phi}, in our situation.
When identifying the free $S$-module $V\otimes S$ with $S^k$ 
for some $k$-dimensional $\C$-vector space $V$,
we will write the element $v\otimes s$ as $s\cdot v$.  

\begin{lemma}
\label{lem:PhiFormula}
If $A$ is the cohomology algebra of $P\Sigma_n^+$, then 
the $S$-linear map $\Phi\colon I^2\otimes S\rightarrow E^3\otimes S$ 
is given by 
\begin{equation}
\label{eq:PhiFormula}
\Phi(r_{ijk}^*)= -(x_{ik}+x_{jk})\cdot u_{jk}u_{ik}u_{ij}+
\sum_{s>t, \{s,t\}\nsubseteq\{i,j,k\}} x_{st}\cdot u_{st}(u_{jk}-u_{ik})u_{ij}\, .
\end{equation}
\end{lemma}
 
\begin{proof}
Recall that $\Phi$ is the composition of the differential $d^2\colon E^2\to E^3$ 
from \eqref{eq:AomotoDiff} with the inclusion $\iota\colon I^2\to E^2$. Hence,
\begin{equation}
\Phi(r_{ijk}^*)=d^2(r_{ijk}^*\otimes 1)=\sum\limits_{1\leq t<s\leq n}^nu_{st} r_{ijk}^*\otimes x_{st} 
=\sum\limits_{1\leq t<s\leq n}^nu_{st} u_{ij}(u_{ik}-u_{jk})\otimes x_{st}  \, .
\end{equation}
Simplifying the last expression using graded-commutativity yields \eqref{eq:PhiFormula}.
\end{proof}

From formula \eqref{eq:PhiFormula}, we see that each entry 
of the matrix of $\Phi$ is of the form $x_{ik}+x_{jk}$ or $x_{st}$ for 
$\{s,t\}\nsubseteq\{i,j,k\}$, $t<s$ and $k<j<i$.

\subsection{A reduced presentation for $\fB_n$} 
\label{subsec:redpres}

By Lemma \ref{lem:presentationAlex}, the $S$-module 
$\fB_n=\fB(P\Sigma^+_n)$ has presentation 
\begin{equation}
 \label{eq:Phistar}
 \xymatrix{
(E^3)^* \otimes S\ar[r]^{{\Phi^*}} & (I^2)^* \otimes S\ar[r] & \fB_n
 }.
\end{equation}
Our next objective is to simplify this presentation in order 
to make it more manageable.
Let $\{r_{ijk} \mid 1\leq k<j<i\leq n\}$ be the basis of the vector space 
$(I^2)^*$, dual to the basis of $I^2$ from \eqref{eq:basisI2}.

 \begin{lemma}
 \label{lem:basisIJK}
 The submodule $\im \Phi^*$ of $(I^2)^*\otimes S$ 
 is generated by the set $\cB=\bigcup\cB_{ijk}$, where the union is over 
 all $1\leq k<j<i\leq n$, and each subset $\cB_{ijk}$ consists of the 
 following elements:
 \begin{equation}
 \label{eq:Rijk}
 \begin{array}{ll}
 \bg_{1}:=(-x_{jk}-x_{l_2k})\cdot r_{ijl_2}+x_{jl_2}\cdot{r}_{ijk} &\\[3pt]
 \bg_{2}:=x_{jk}\cdot r_{ijl_2}+x_{il_2}\cdot{r}_{ijk} &\\[3pt]
 \bg_{3}:=-x_{jk}\cdot r_{il_3j}+x_{il_3}\cdot{r}_{ijk} &\\[3pt]
 \bg_{4}:=x_{jk}\cdot r_{l_4ij}+x_{l_4i}\cdot{r}_{ijk} &\\ [3pt]
  \bh_{1}:=({x}_{il_1}+{x}_{jl_1}+{x}_{kl_1})\cdot{r}_{ijk} &\\[3pt]
 \bh_{2}:=({x}_{ik}+x_{jk})\cdot{r}_{ijk} &\\
 \end{array} \qquad\qquad
 \begin{array}{ll}
 \bh_{3}:=x_{l_{2}k}\cdot{r}_{ijk} &\\[3pt]
 \bh_{4}:=x_{l_3k}\cdot{r}_{ijk} &\\[3pt]
 \bh_{5}:=x_{l_3j}\cdot{r}_{ijk} &\\[3pt]
 \bh_{6}:=x_{l_4k}\cdot{r}_{ijk} &\\[3pt]
 \bh_{7}:=x_{l_4j}\cdot{r}_{ijk} &\\[3pt]
 \bh_{8}:=x_{st}\cdot{r}_{ijk} &\\
 \end{array} 
  \end{equation}
 where $1\leq l_1 <k<l_2<j<l_3<i<l_4\leq n$ and $\{s,t\}\cap\{i,j,k\}=\emptyset$.%
\footnote{Here and in the sequel, a symbol such as $\bg_{m}$ or $\bh_{m}$ 
denotes a single polynomial, which depends on $m$, but also on the indices 
$i,j,k$, and some of $l_1,l_2,l_3,l_4,s,t$. To avoid a plethora of such indices, 
we will omit them as much possible from the notation, whenever the context 
makes it clear what they are.}
\end{lemma}

\begin{proof}
Write $\Phi^*_q$ for the restriction of $\Phi^*$ to the subspace spanned 
by the basis vectors of cardinality $q:=\sharp \{i,j,k,l,s,t\}$. 
The map $\Phi^*$ can then be decomposed as the block-matrix 
$\Phi^*_3\oplus \Phi^*_4 \oplus \Phi^*_5\oplus \Phi^*_6$.
We now analyze formula \eqref{eq:PhiFormula} case by case, 
according to the cardinality $q=3,4,5,6$.

When $q= 3$, we have $l=i,s=j, t=k$. 
Then $\Phi^*_3((u_{jk}u_{ik}u_{ij})^*)=-(x_{ik}+x_{jk})\cdot r_{ijk}$, 
and so $\Phi^*_3$ contributes elements of the form $\bh_2$ to $\cB$.

When $q= 4$, suppose $i>j>k>l$. 
There are then $\binom{6}{3}-\binom{4}{3}=16$ possible 
combinations:
\begin{equation*}
\label{eq:Psi4}
\begin{array}{l@{\hspace{10pt}}l}
\Phi^*_4(({u}_{kl}{u}_{jl}{u}_{il})^*)=0\\
\Phi^*_4(({u}_{kl}{u}_{jl}{u}_{ik})^*)={-{x}_{jl}\cdot{r}_{ikl}}\\
\Phi^*_4(({u}_{kl}{u}_{jl}{u}_{ij})^*)={x}_{kl}\cdot{r}_{ijl}\\
\Phi^*_4(({u}_{kl}{u}_{jk}{u}_{il})^*)={x}_{il}\cdot{r}_{jkl}\\
\Phi^*_4(({u}_{kl}{u}_{jk}{u}_{ik})^*)=-{x}_{jk}\cdot{r}_{ikl}+{x}_{ik}\cdot{r}_{jkl}\\
\Phi^*_4(({u}_{kl}{u}_{jk}{u}_{ij})^*)={x}_{kl}\cdot{r}_{ijk}+{x}_{ij}\cdot{r}_{jkl}\\
\Phi^*_4(({u}_{kl}{u}_{il}{u}_{ij})^*)={-{x}_{kl}\cdot{r}_{ijl}}\\
\Phi^*_4(({u}_{kl}{u}_{ik}{u}_{ij})^*)=-{x}_{kl}\cdot{r}_{ijk}+{x}_{ij}\cdot{r}_{ikl}\\
\end{array}
\begin{array}{ll}
\Phi^*_4(({u}_{jl}{u}_{jk}{u}_{il})^*)={-{x}_{il}\cdot{r}_{jkl}}\\
\Phi^*_4(({u}_{jl}{u}_{jk}{u}_{ik})^*)={-{x}_{ik}\cdot{r}_{jkl}}\\
\Phi^*_4(({u}_{jl}{u}_{jk}{u}_{ij})^*)={x}_{jl}\cdot{r}_{ijk}-{x}_{jk}\cdot{r}_{ijl}-{x}_{ij}\cdot{r}_{jkl}\\
\Phi^*_4(({u}_{jl}{u}_{il}{u}_{ik})^*)={-{x}_{jl}\cdot{r}_{ikl}}\\
\Phi^*_4(({u}_{jl}{u}_{ik}{u}_{ij})^*)=-{x}_{jl}\cdot{r}_{ijk}-{x}_{ik}\cdot{r}_{ijl}\\
\Phi^*_4(({u}_{jk}{u}_{il}{u}_{ik})^*)={-{x}_{jk}\cdot{r}_{ikl}}\\
\Phi^*_4(({u}_{jk}{u}_{il}{u}_{ij})^*)=-{x}_{il}\cdot{r}_{ijk}-{x}_{jk}\cdot{r}_{ijl}\\
\Phi^*_4(({u}_{il}{u}_{ik}{u}_{ij})^*)=-{x}_{il}\cdot{r}_{ijk}+{x}_{ik}\cdot{r}_{ijl}-{x}_{ij}\cdot{r}_{ikl}\\
\end{array}
\end{equation*}
 The image of $\Phi^*_4$ is generated by $({x}_{il}+{x}_{jl}+{x}_{kl})\cdot{r}_{ijk}$ and 
 the elements
\begin{equation*}
\label{eq:Psi4reduced}
\begin{cases}
{x}_{kl}\cdot{r}_{ijl}\\
(-{x}_{jl}-{x}_{kl})\cdot{r}_{ijk}+{x}_{jk}\cdot{r}_{ijl}\\
{x}_{jl}\cdot{r}_{ijk}+{x}_{ik}\cdot{r}_{ijl}\\
\end{cases} 
\begin{cases}
{x}_{jl}\cdot{r}_{ikl}\\
{x}_{jk}\cdot{r}_{ikl}\\
-{x}_{kl}\cdot{r}_{ijk}+{x}_{ij}\cdot{r}_{ikl}\\
\end{cases} 
\begin{cases}
{x}_{il}\cdot{r}_{jkl}\\
{x}_{ik}\cdot{r}_{jkl}\\
{x}_{kl}\cdot{r}_{ijk}+{x}_{ij}\cdot{r}_{jkl}\\
\end{cases} 
 \end{equation*}
Hence, the image of $\Phi^*_4$ contributes $\bh_1=({x}_{il_1}+{x}_{jl_1}+{x}_{kl_1})\cdot{r}_{ijk}$
for $l_1\leq k-1$, as well as  
 $\bg_{1}$, $\bg_{2}$, $\bh_{3}$
for $k<l_2<j<i$,
$\bg_3$, $\bh_4$, $\bh_5$ for $k<j<l_3<i$, 
and $\bg_4$, $\bh_6$, $\bh_7$ for $k<j<i<l_4$.

When $q= 5$, the only possible instance for which $\Phi^*_5\neq 0$ is when 
 $l=j$, or $l=i$, or $s=k$, or $s=l$. Suppose  $u_{ij}>u_{lk}>u_{st}$. 
 Using formula \eqref{eq:PhiFormula} again,
 we find that
\begin{equation*}
\Phi^*_5((u_{st}u_{lk}u_{ij})^*)=\left\{
\begin{array}{ll}
x_{st}\cdot r_{ijk} & \text{if $l=j$}\\
-x_{st}\cdot r_{ijk} & \text{if $l=i$}\\
x_{ij}\cdot r_{lkt} & \text{if $s=k$}\\
-x_{ij}\cdot r_{lkt}  & \text{if $s=l$}\\
0 & \text{otherwise}.
\end{array}
\right.
\end{equation*}
Hence, the map $\Phi^*_5$ will contribute $\bh_{8}=x_{st}\cdot r_{ijk}$ 
for $\{s,t\}\cap\{i,j,k\}=\emptyset$ to $\cB$.
 
When $q= 6$, we have that $\Phi^*_6((u_{st}u_{lk}u_{ij})^*)=0$.
This completes the proof.
\end{proof} 

Let us denote by $m_{ijk}$ and $m$ the cardinalities of the sets $\cB_{ijk}$ 
and $\cB$, respectively. Clearly, $m_{ijk}= \binom{n}{2}-2k$. An elementary 
computation shows that 
\begin{align}
\label{eq:bsize}
m=\sum_{n\geq i>j>k\geq 1} m_{ijk}= \sum_{k=1}^{n-2} \binom{n-k}{2} m_{ijk}=
\frac{1}{12}n (n^4 - 5 n^3 + 7 n^2 - n  -2  ).
\end{align}

Let $S^m$ be the free $S$-module generated by the set $\cB=\bigcup\cB_{ijk}$, 
endowed with the subset order defined by setting 
$\cB_{lst}{\succ}\cB_{ijk}$ if either $i>l$, or $i=l$ and $j>s$, or $i=l$, $j=s$, and $k>t$. 
For elements in each $\cB_{ijk}$, we use the order
defined by the coefficients of $r_{ijk}$ by setting 
\begin{equation}
\label{eq:orderx}
x_{st}{\succ}x_{kl} \textrm{ if either }
s>k, \textrm{ or } s=k \textrm{ and } t>l.
\end{equation}
Together with Lemmas \ref{lem:presentationAlex} and \ref{lem:basisIJK}, we obtain the desired 
presentation for the $S$-module $\fB_n$. 

\begin{prop}
\label{prop:reducedPres}
The infinitesimal Alexander invariant $\fB_n=\fB(P\Sigma^+_n)$ admits a 
minimal presentation of the form
\begin{equation}
\label{eq:PresentationPsi}
\xymatrix{
 S^{m}\ar[r]^{{\Psi}} &S^{\binom{n}{3}}\ar[r] & \fB_n.
 }
\end{equation}
The matrix of ${\Psi}$ is upper block triangular, 
with diagonal row vectors $\vec{v}_{ijk}$ ($1\leq k<j<i\leq n$) given by  
\[
(\vec{v}_{ijk})_* =
\begin{cases}
x_{il}+x_{jl}+x_{kl} & \text{for\, $1\leq l\leq k-1$},  \\
x_{ik}+x_{jk}         & \text{for\, $1\leq k<j<i\leq n$}, \\
x_{st}                    & \text{for\, $\{s,t\}\not\subset \{i,j,k,l\}$ and $1\leq l\leq k-1$}.
\end{cases}
\]
\end{prop}

\begin{proof}
From Lemma \ref{lem:presentationAlex}, we know that the standard basis for 
$(I^2)^*\otimes S=S^{\binom{n}{3}}$ gives a minimal generating set 
for $\fB_n$. From Lemma \ref{lem:basisIJK}, the submodule 
$\im \Phi^*\subset (I^2)^*\otimes S$ is generated by the independent set $\cB$. 
Hence, the presentation \eqref{eq:PresentationPsi} has no redundant relations. 
\end{proof}

\begin{example}
\label{ex:Omega4}
The first non-trivial example is the $S$-module $\fB_4=\fB(P\Sigma^+_4)$.  
Applying Proposition \ref{prop:reducedPres}, we find that 
$\fB_4=\coker (\Psi\colon S^{14}\to S^4)$, 
where the transpose of the matrix of $\Psi$ has the form
\begin{equation*}
\begin{pmatrix}
\vec{v}_{432}^T&0&0&0\\
* &\vec{v}_{431}^T&0&0\\
* &*&\vec{v}_{421}^T&0\\
* &*&*&\vec{v}_{321}^T\\
\end{pmatrix}
=
\begin{pmatrix}
{x}_{41}+{x}_{31}+{x}_{21}&0&0&0\\
{x}_{42}+{x}_{32}&0&0&0\\
0&{x}_{21}&0&0\\
-{x}_{31}-{x}_{21}&{x}_{32}&0&0\\
0&{x}_{41}+{x}_{31}&0&0\\
{x}_{31}&{x}_{42}&0&0\\
0&0&{x}_{31}&0\\
0&0&{x}_{32}&0\\
0&0&{x}_{41}+{x}_{21}&0\\
{-{x}_{21}}&0&{x}_{43}&0\\
0&0&0&{x}_{31}+{x}_{21}\\
0&0&0&{x}_{41}\\
0&0&0&{x}_{42}\\
{x}_{21}&0&0&{x}_{43}\\
\end{pmatrix}.
\end{equation*}
 \end{example}
 
\section{A Gr\"{o}bner basis for $\fB(P\Sigma_n^+)$}
\label{sec:grobnerbasis}

In this section, we determine a Gr\"{o}bner basis for the infinitesimal
Alexander invariant of $P\Sigma_n^+$, which will play a crucial role 
in computing the Chen ranks and the scheme structure of the first 
resonance varieties of the upper McCool groups.

\subsection{Gr\"{o}bner basis for modules}
\label{subsec:gb review}

We start by recalling some background material on Gr\"{o}bner 
basis for modules (see \cite[\S 15]{Eisenbud95} for details).  
Let $S=\C[\x]$ be a polynomial ring with variables in a finite set $\x$, 
and let $F$ be a free $S$-module 
with basis $\{e_1,\dots,e_r\}$. A \textit{monomial}\/ in $F$ is an element of form 
$m=\x^{\alpha}e_i$ and a \textit{term}\/ in $F$ is an element of the form 
$c\cdot\x^{\alpha}e_i$, where $c\in\C$.   A \textit{monomial order}\/ on 
$F$ is a total order $\succ$ on the monomials of $F$ such that if 
$m_1$ and $m_2$ are monomials in $F$ and $s\neq 1$ is a monomial 
in $S$, then $m_1\succ m_2$ implies $sm_1\succ sm_2\succ m_2$.

Given a monomial order $\succ$ on $F$, the \textit{initial term}\/ 
of an element $f\in F$ is the largest term of $f$ with respect to $\succ$, 
denoted by $\ini_{\succ}(f)$.  For a submodule $I\subset F$, we let 
$\ini_{\succ} (I)$ denote the submodule generated by $\{\ini_{\succ}(f)\mid f\in I\}$. 
A set $\{g_1,\dots, g_s\}$ is called a \textit{Gr\"{o}bner basis}\/ for the module $I$ 
if the elements $g_1,\dots,g_s$ generate $I$, while at the same time 
$\ini_{\succ}(g_1),\dots,\ini_{\succ}(g_s)$ generate $\ini_{\succ}(I)$.

If the initial terms $\ini_{\succ}(g_i)$ and $\ini_{\succ}(g_j)$ 
contain the same basis element $e_i$ of $F$, put 
\begin{equation}
\fS(g_i,g_j):=\dfrac{\ini_>(g_j)}{\gcd(\ini_{\succ}(g_i),\ini_{\succ}(g_j))}\cdot g_i
-\dfrac{\ini_{\succ}(g_i)}{\gcd(\ini_{\succ}(g_i),\ini_{\succ}(g_j))}\cdot g_j .
\end{equation}
Using the division algorithm, the element $\fS(g_i,g_j)\in F$ has a standard expression
of the form 
\begin{equation}
\label{eq:six}
\fS(g_i,g_j)=\sum p_k^{ij}\cdot g_k+h_{ij},
\end{equation}
where $p_k^{ij}\in S$ and $\ini_{\succ}(p_{k}^{ij}g_k)\prec \lcm(\ini_{\succ}(g_i),\ini_{\succ}(g_j))$.
If $\ini_{\succ}(g_i)$ and $\ini_{\succ}(g_j)$ 
contain distinct basis elements of $F$, we set $h_{ij}=0$.  
Buchberger's criterion asserts that the set $\{g_1,\dots, g_t\}$ is a Gr\"{o}bner basis 
for the ideal $I$ if and only if all $\fS$-polynomials $\fS(g_i,g_j)$ \emph{vanish}, i.e.,  
$h_{ij}=0$ for all $i$ and $j$.
 
\subsection{A Gr\"{o}bner basis for $\fB_n$}
\label{subsec:GBAlexInv}

Once again, let $S=\C[\x]$ be the coordinate ring of $H^1(P\Sigma_n^+;\C)$
with variables ordered as in \eqref{eq:orderx}.
Recall from Proposition \ref{prop:reducedPres} that 
$\fB_n=\coker (\Psi)$, where  $\Psi$ is an $S$-linear map $\Psi\colon S^{m} \to (I^2)^*\otimes S$.
Let us order the basis of $(I^2)^*\otimes S$ by setting 
\begin{equation}
\label{eq:orderr}
\textrm{ $r_{lst}{\succ}r_{ijk}$  if either $i>l$, or $i=l$ and $j>s$, 
or $i=l$, $j=s$, and $k>t$.  }
\end{equation}
We use the \emph{graded reverse lexicographic} order on $S$ defined by 
$\x^{\alpha}\succ \x^{\beta}$ if $\deg(\x^{\alpha})>\deg(\x^{\beta})$, or 
$\deg(\x^{\alpha})=\deg(\x^{\beta})$ and the right-most entry in $\alpha-\beta$ is negative. 
This order on $S$ is extended to a monomial order on $(I^2)^*\otimes S$
by declaring $\x^{\alpha} r_{lst}\succ \x^{\beta}r_{ijk}$ if $r_{lst}{\succ}r_{ijk}$, or if $r_{lst}{=}r_{ijk}$ and 
$\x^{\alpha}\succ \x^{\beta}$. 
 
By Lemma \ref{lem:basisIJK}, the module $\im(\Psi)$ is generated 
by the set $\cB=\bigcup_{1\leq k<j<i\leq n} \cB_{ijk}$, where $\cB_{ijk}$ consists 
of the elements from \eqref{eq:Rijk}.

\begin{theorem}
\label{thm:GrobnerBasis}
A Gr\"{o}bner basis for the $S$-module 
$\im(\Psi)$ is given by $\cG=\bigcup_{1\leq k<j<i\leq n} \cG_{ijk}$, where
$\mathcal{G}_{ijk}$ is the union of $\mathcal{B}_{ijk}$ and
\begin{equation}
\label{eq:dijk}
\mathcal{D}_{ijk}:=\{\bh_9:=x_{kl}x_{kp}\cdot r_{ijk}, \: \bh_{0}:=x_{jq}x_{kp}\cdot r_{ijk} 
\mid  1\leq p\leq l<  k, 1\leq q\leq k\}.
\end{equation}
\end{theorem}

The proof of this theorem is standard but lengthy, as it involves checking that $\cG$ 
generates the $S$-module $\im(\Psi)$ as a submodule of $(I^2)^*\otimes S$,  
and all $\fS$-polynomials of elements in $\cG$ vanish. We thus relegate the proof to 
Appendix \ref{sec:Appendix}.  
 
\begin{corollary}
\label{cor:diag}
The above Gr\"{o}bner basis  $\cG$ for $\im(\Psi)$ admits an upper block 
triangular matrix with diagonal row vectors $\vec{w}_{ijk}$ for $1\leq k<j<i\leq n$, 
where each vector $\vec{w}_{ijk}$ is constructed from the vector $\vec{v}_{ijk}$ 
from Proposition \ref{prop:reducedPres} by adding entries 
$\{x_{kl}x_{ks}, x_{jt}x_{ks} \mid 1\leq s\leq l\leq  k-1, 1\leq t\leq k\}$.
Furthermore, the vector $\vec{w}_{ijk}$ has  
$\binom{n}{2}+\binom{k}{2}+(k-3)k$ entries.
\end{corollary}

\begin{proof}
The first assertion is clear. 
The length of the vector $\vec{w}_{ijk}$ 
is computed by counting the (linear) entries 
in the vector $\vec{v}_{ijk}$ from Proposition \ref{prop:reducedPres}, 
and adding the number of quadratic entries.
\end{proof}
 
\section{The Chen ranks of the upper McCool groups}
\label{sec:ChenRanks}

In this section, we compute the Hilbert series of the infinitesimal Alexander 
invariants of the upper McCool groups.  We then use this information to compute 
the Chen ranks of $P\Sigma_n^+$ and answer a question from \cite{Cohen-P-V-Wu08}.

\subsection{Hilbert series of monomial ideals}
\label{subsec:hilb monomial}

We first review some background from \cite[\S15.1]{Eisenbud95}. 
Let $S$ be a polynomial ring. By a standard result in commutative algebra,  
the computation of the Hilbert series of any finitely generated, graded $S$-module $M$ can 
be reduced to the computation of the Hilbert series of a monomial module.  More precisely, 
write $M=S^n/\fI$, where $\fI$ is a submodule generated by homogeneous elements in $S^n$; 
then 
\begin{equation}
\label{eq:HilbertSeries}
\Hilb(S^n/\fI,t)=\Hilb(S^n/\ini(\fI),t)\, ,
\end{equation}
where  $\ini(\fI)$ is the submodule generated by the initial terms of $\fI$.
Since the Hilbert function is additive, we only need to treat the 
case $N=S/I$, where $I$ is a monomial ideal of $S$.

Let $\{m_1,\dots,m_t\}$ be a set of monomials generating $I$. 
Choose a monomial $p\in F$, and denote its degree by $d$. 
Let $J$ be the monomial ideal generated by $\{p, m_1,\dots, m_t\}$, 
and let $I^{\prime}$ 
be the ideal generated by $\{m_1/\gcd(m_1,p),\dots, m_t/\gcd(m_t,p)\}$. 
(Certain choices of monomials $p$ can produce ideals $I^{\prime}$ and 
$J$ generated by fewer monomials, in less variables.)
We then have a short exact sequence of graded $S$-modules, 
\begin{equation}
\label{eq:exactseq}
\xymatrixcolsep{18pt}
\xymatrix{0\ar[r]& S/I^{\prime}(-d)\ar[r]&  S/I\ar[r]&  S/J\ar[r]&  0}.
\end{equation}
Taking Hilbert series, the following equality holds:
 \begin{equation}
 \label{eq:MonomialHilbert}
  \Hilb(S/I,t)=\Hilb(S/J,t)+t^{d}\Hilb(S/I^{\prime},t). 
 \end{equation}  
 
\subsection{The Hilbert series of $\fB_n$}
\label{subsec:HilbAlexInv}

We are now ready to compute the Hilbert series of the infinitesimal 
Alexander invariants $\fB_n$ of the upper McCool groups $P\Sigma^+_n$.
\begin{theorem} 
\label{thm:HilbertSeries}
The Hilbert series of the $S$-module $\fB_n$ is given by
\begin{equation}
\label{eq:HilbertSeriesBn}
\Hilb(\fB_n,t)
=\sum\limits_{s=2}^{n-1}\binom{s}{2}\frac{1}{(1-t)^{n-s+1}}+\binom{n}{4}\frac{t}{1-t}.
\end{equation}
\end{theorem}
 
\begin{proof}
This computation is an application of the method from \cite[\S15.1.1]{Eisenbud95}.
Since we already found a Gr\"{o}bner basis $\cG$ for $\fB_n=\im(\Psi)$, 
formula \eqref{eq:HilbertSeries} insures that 
we only need to compute the Hilbert series of the resulting monomial ideal, 
$\ini_{\succ}(\im(\Psi))=\langle \ini_{\succ}(\cG)\rangle$.

Recall from Theorem \ref{thm:GrobnerBasis} and Lemma \ref{lem:basisIJK}
that 
\begin{equation}
\label{eq:ini gijk}
\ini_{\succ}(\cG_{ijk})=
\left\{ 
\begin{array}{ll}
x_{ks}x_{kl}\cdot r_{ijk},  ~x_{jt}x_{kl}\cdot r_{ijk}, \\
x_{ik}\cdot r_{ijk},   ~  x_{il}\cdot r_{ijk},  ~x_{ab}\cdot r_{ijk}   \\
\end{array}
\middle|
\begin{array}{ll}
1\leq l\leq s\leq k-1, 1\leq t\leq k,\\
 \{a,b\}\not\subset \{i,j,k,l\}\\
\end{array}
\right\}.
\end{equation}
Consider the (reduced) monomial ideal 
\begin{equation}
\label{eq:idijk}
I_{ijk}=\langle
x_{ks}x_{kl},\  x_{jm}x_{kl}\ (1\le l\le s\le k-1, 1\le m\le k), \ 
x_{ik}, \  x_{il}, \ x_{ab} \ (\{a,b\}\not\subset \{i,j,k,l\} )
\rangle.
\end{equation}
Using \eqref{eq:MonomialHilbert}, a straightforward 
computation shows that the Hilbert series of this  
ideal is given by $\Hilb(S/I_{ijk}, t)= 1/(1-t)^k+ kt/(1-t)$.
Hence, the Hilbert series of $\fB_n$ is given by
\begin{equation}
\label{eq:hilb fbn}
\Hilb(\fB_n,t)=\sum_{i>j>k}\Hilb(S/ I_{ijk},t)
 =\sum_{k=1}^{n-2}\binom{n-k}{2}\left( \dfrac{1}{(1-t)^{k+1}}+\dfrac{kt}{1-t}\right).
\end{equation}
Upon setting $s=n-k$, the claimed formula follows at once. 
\end{proof}

\subsection{The Chen ranks of $P\Sigma_n^+$}
\label{subsec:chen ranks}
With Theorem \ref{thm:HilbertSeries} at our disposal,  
we may now  compute the Chen ranks of the 
upper McCool groups $P\Sigma_n^+$, for all $n\ge 1$.

\begin{theorem}
\label{thm:chenranks}
The Chen ranks $\theta_k=\theta_k(P\Sigma_n^+)$ are given by
$\theta_1=\binom{n}{2}$, $\theta_2=\binom{n}{3}$, $\theta_3=2\binom{n+1}{4}$, and
\begin{equation*}
\theta_k=\binom{n+k-2}{k+1}+\theta_{k-1}=\sum\limits_{i=3}^k\binom{n+i-2}{i+1}+ \binom{n+1}{4}
\end{equation*}
for $k\geq 4$.
\end{theorem}
  
\begin{proof}
Clearly, $\theta_1(P\Sigma_n^+)=b_1(P\Sigma_n^+)=\binom{n}{2}$.  
To compute the other Chen ranks, recall from \eqref{eq:ChenHilbert} that 
$\sum_{k\geq0}\theta_{k+2}(P\Sigma_n^+)\cdot t^k={\rm Hilb}(\fB(P\Sigma_n^+),t)$. 
On the other hand, Theorem \ref{thm:HilbertSeries} provides 
an expression for the Hilbert series of the infinitesimal 
Alexander invariant $\fB_n=\fB(P\Sigma_n^+)$.   
Thus, it remains to find the coefficient of $t^k$ on the 
right-hand side of \eqref{eq:HilbertSeriesBn}.  Let 
\begin{equation}
\label{eq:ft1}
f(t)=\sum\limits_{s=2}^{n-1}\binom{s}{2}(1-t)^{-n+s-1}+\binom{n}{4} \, t(1-t)^{-1}. 
\end{equation}
Computing derivatives, we find that 
\begin{equation}
\label{eq:ft2}
f^{(k)}(t)= \sum\limits_{s=2}^{n-1}\binom{s}{2}\prod_{i=1}^k (n-s+i)(1-t)^{-n+s-k-1}+k!
\binom{n}{4}(1-t)^{-k-1}.
\end{equation}
Hence, the Chen ranks of $P\Sigma_n^+$ are given by
\begin{equation}
\label{eq:ft3}
\theta_{k+2}=\frac{1}{k!}f^{(k)}(0)= \sum\limits_{s=2}^{n-1}
\binom{s}{2}\prod_{i=1}^k (n-s+i)+k!\binom{n}{4}.
\end{equation}
Simplifying this expression, we obtain the claimed recurrence formula.
\end{proof}

\subsection{Distinguishing some related groups}
\label{subsec:discuss}
Both the pure braid groups $P_n$ and the upper McCool groups $P\Sigma_n^+$ 
are iterated semidirect products of the form $F_{n-1}\rtimes \dots \rtimes F_2\rtimes F_1$. 
Clearly, $P_1=P\Sigma_1^+ =\{1\}$ and $P_2=P\Sigma_2^+ =\Z$; it is also known that 
$P_3\cong P\Sigma_3^+\cong F_2\times F_1$.  Furthermore, 
both $P_n$ and $P\Sigma_n^+$ share the same LCS ranks and the same Betti numbers
as the corresponding direct product of free groups, $\Pi_n= \prod_{i=1}^{n-1}F_{i}$, see
\cite{Arnold69, Cohen-P-V-Wu08, FalkRandell, Kohno85}. In \cite{Cohen-P-V-Wu08}, 
F.~Cohen et al.~asked whether the groups $P_n$ and  $P\Sigma_n^+$ are 
isomorphic, for $n\ge 4$.  The next corollary answers this question.

\begin{corollary}
\label{cor:chen ppp}
For each $n\geq 4$, 
the pure braid group $P_n$, the upper McCool group $P\Sigma_n^+$, 
and the product group $\Pi_n$ are pairwise 
non-isomorphic.
\end{corollary}

\begin{proof}
As shown in \cite{Cohen-Suciu95}, the fourth Chen ranks of $P_n$  
and  $\Pi_n$ are given by
$\theta_4(P_n)=3\binom{n+1}{4}$ and 
$\theta_4(\Pi_n)=3\binom{n+2}{5}$, respectively.   
On the other hand, from Theorem \ref{thm:chenranks}, we have that 
\begin{equation}
\label{eq:chen psigman}
\theta_{4}(P\Sigma_n^+)= 2\binom{n+1}{4}+\binom{n+2}{5}.  
\end{equation}
Comparing these ranks shows that the groups $P_n$, $\Pi_n$, and 
$P\Sigma_n^+$ have non-isomorphic maximal metabelian quotients, 
and thus are pairwise non-isomorphic.
\end{proof}

In \cite{Bardakov-Mikhailov08}, Bardakov and Mikhailov attempted to 
prove that $P_4$ is not isomorphic to $P\Sigma_4^+$ by showing that
these two groups have different single-variable Alexander polynomials. 
To explain their approach (and why it does not work), consider 
a finitely presented group $G$, and let $H=G_{\ab}/{\rm torsion}$ 
be the maximal torsion-free abelian quotient of $G$. 
The group ring $R=\Z{H}$ is a Noetherian, commutative ring 
and a unique factorization domain.  Moreover, the $R$-module  
$\overline{B}(G):=B(G)\otimes_{\Z{G_{\ab}}} \Z{H}$ 
is finitely presented;  let $E_0(\overline{B}(G))$ 
be the ideal of maximal minors of a presentation matrix for this 
module.  The Alexander polynomial of $G$, then, is 
the greatest common divisor (gcd) of all elements 
of $E_0(\overline{B}(G))$;  this polynomial, denoted $\Delta_G$, 
is well-defined up to units in $R$. 

Now let $\phi\colon G\to \Z$ be a homomorphism, and denote by 
$\phi\colon \Z{G}\to \Z{\Z}$ its extension to group rings.  Identifying 
$\Z{\Z}=\Z[t^{\pm 1}]$ and letting  
$B(G)^{\phi}:=\overline{B}(G)\otimes_{\Z{H}} \Z[t^{\pm 1}]$ be the 
corresponding $\Z[t^{\pm 1}]$-module, the single-variable 
Alexander polynomial of $G$ with respect to $\phi$, denoted by 
$\Delta^{\phi}_G(t)$, is the gcd of all elements of $E_0(B(G)^{\phi})$.

\begin{example}
\label{ex:P3}
Let $P_3=\langle x_1,x_2,x_3 \mid x_1x_2x_3~{\rm central}\rangle$ 
be the pure braid group on $3$ strands.   Letting 
$\phi\colon P_3\to \Z$ be the homomorphism given by $\phi(x_i)=t$, 
we find that $\Delta^{\phi}_{P_3}(t)=(1-t^3)(1-t)$. On the other hand, if we take 
the presentation $P_3=\langle x_1,x_2,z \mid z~{\rm central}\rangle$, 
and let $\psi\colon P_3\to \Z$ be the homomorphism given by 
$\psi(x_1)=\psi(x_2)=\psi(z)=t$, then $\Delta^{\psi}_{P_3}(t)=(1-t)^2$.
\end{example}

This example shows that the single-variable Alexander 
polynomial of a finitely presented group $G$ depends on a choice of 
presentation for the group, and thus is not an isomorphism-type invariant.  
Hence, the argument from \cite{Bardakov-Mikhailov08} does {\em not}\/ 
rule out the existence of an isomorphism  $P_4\cong P\Sigma_4^+$.  
On the other hand, the (multi-variable) Alexander polynomial $\Delta_G$ 
{\em is}\/ an  isomorphism-type invariant for finitely presented groups $G$.  
Nevertheless, the groups $P_4$ and $P\Sigma_4^+$ cannot be distinguished 
by means of the multi-variable Alexander polynomial.  Indeed, 
it is known that $\Delta_{P_n}=1$, for all $n\ge 4$ (see \cite[Theorem 9.15]{Suciu11}), 
while direct computation shows that $\Delta_{P\Sigma_4^+}=1$, too.

\section{Resonance varieties and resonance schemes}
\label{sect:resonance scheme}

We start this section with a quick review of the resonance varieties 
of a connected, locally finite, graded, graded-commutative algebra.  
We then discuss the natural scheme structure of these varieties, 
and give a quick introduction to the Chen ranks formula.  

 \subsection{Resonance varieties}
\label{subsec:resonance}
Let $V$ be a complex vector space of finite dimension, and let $V^*$ be its dual. 
We write $E=\bigwedge V$ for the exterior algebra on $V$, and $S=\Sym(V^*)$ for the 
symmetric algebra on $V^*$.  
Let $\{e_1,\dots,e_n\}$ and $\{x_1,\dots,x_n\}$ be dual bases for $V$ and 
$V^*$, respectively, and identify the symmetric algebra $\Sym(V^*)$ 
with the polynomial ring $S=\C[x_1,\dots,x_n]$.   
 
Now let $A$ be graded, graded-commutative 
$\C$-algebra; we will assume that $A$ is connected and locally finite. 
The (degree $i$, depth $d$) \emph{resonance varieties}\/ 
of the graded algebra $A$ are the homogeneous algebraic subvarieties 
of the affine space $A^1$ defined as 
\begin{equation}
\label{eq:resvar}
\cR^i_d(A)=\big\{ a\in A^1\mid \dim_{\C}H^i(A, \delta_a)\geq d \big\}.
\end{equation}
where $(A,\delta_a)$ is the cochain complex (known as the {\em Aomoto complex}) 
with differentials $\delta^i_{a}\colon A^i\to A^{i+1}$ given by 
$\delta^i_{a}(u)=a\cdot u$.  
According to \cite{Papadima-Suciu10, Suciu12}, the evaluation of 
the cochain complex \eqref{eq:universalAomoto} at an element $a\in A^1$ 
coincides with the Aomoto complex $(A,\delta_a)$.
 
When $A=E$ is an exterior algebra, the 
Aomoto complex $(E,\delta_a)$ is acyclic, for each non-zero 
element $a\in E^1$, and thus $\cR^i_d(E)\subseteq \{0\}$, 
for all $i$ and $t$.  In general, though, the resonance varieties 
of a graded algebra $A$ can be arbitrarily complicated.  
For more details on this subject, we refer to 
\cite{Cohen-Suciu99, Matei-Suciu00, Suciu01, 
Papadima-Suciu06, Dimca-Papadima-Suciu09, Suciu11, 
Suciu12, Papadima-Suciu15, Cohen-Schenck15}, 
and references therein.

We will focus in this paper on the degree-$1$ resonance varieties, 
$\cR_d(A):=\cR^1_d(A)$. These varieties depend only on the multiplication 
map, $\mu_A\colon A^1\wedge A^1\to A^2$, and thus, only on the quadratic 
closure $\bar{A}$, defined in \eqref{eq:quadclose}, i.e., $\cR_d(A)=\cR_d(\bar{A})$. 
Moreover, it is readily seen that 
\begin{equation}
\label{eq:res1}
\RR_{d}(A) =\left\{ a \in A^1 \: \left|\:
\begin{array}
[c]{l}%
\text{there is a linear subspace $W \subset A^1$ of dimension 
$d$} \\ \text{such that $a\notin W$ and $\mu_A(a , b)= 0$ for all $b\in W$}
\end{array}
\right\}\right. .
\end{equation}

Now let $G$ be a finitely generated group, and suppose the 
cohomology algebra $A=H^{\ast}(G;\C)$ is locally finite. 
The resonance varieties of $G$ are then defined as  
$\cR^i_d(G):=\cR^i_d(A)$. Most important to us 
is the first  (depth-$1$) resonance variety,
\begin{equation}
\label{eq:FirstResonance}
\cR_1(G)=\{a\in H^1(G,\C)\mid \exists b \in H^1(G,\C),\, b\ne \lambda a,\, a b=0\}, 
\end{equation}
in which case no further assumption on $G$ 
besides finite generation is needed to insure that 
$\cR_1(G)$ is a Zariski closed set. The following 
(easy to prove) naturality property will be useful in the sequel.

\begin{lemma}[\cite{Papadima-Suciu06}]  
\label{lem:PS06}
Let $G_1$ be a finitely generated group, and let 
$\alpha\colon G_1\rightarrow G_2$ be a surjective homomorphism.  
Then the induced monomorphism in cohomology, 
$\alpha^{\ast}\colon H^1(G_2;\C)\rightarrow H^1(G_1;\C)$, 
takes $\cR_1(G_2)$ to $ \cR_1(G_1)$.
\end{lemma} 
 
\subsection{Resonance schemes}
\label{subsec:resonanceScheme}
Before proceeding, let us review some basic notions from commutative 
algebra and the geometry of schemes, as recounted for instance 
in \cite{Eisenbud95, Eisenbud05}. We work over a polynomial ring 
$S=\C[x_1,\dots ,x_n]$, and denote by $\V(\mathfrak{I})\subset \C^n$ 
the variety defined by an ideal $\mathfrak{I}\subset S$.  

Let $M$ be a finitely generated $S$-module.  Suppose that a minimal 
primary decomposition of the annihilator ideal of $M$ is given by 
\begin{equation}
\label{eq:primarydecom} 
\Ann(M)=\bigcap_{i=1}^p Q_i .
\end{equation}
Let $\fq_i=\sqrt{Q_i}$ be the corresponding radical ideals (or, associated primes).
The varieties $\bV(\fq_i)$ cut out by the ideals $\fq_i$ 
for $1\leq i\leq p$ form the \emph{scheme}\/ $\Spec(S/\Ann(M))$ associated to $M$. 
Geometrically, this scheme consists of {\em isolated}\/ components, 
which are the irreducible components of the support variety $\V(\Ann(M))$,  
and of {\em embedded}\/ components, which are certain subvarieties 
of the isolated components. 

We say that the variety $\V(\Ann(M))$ is \emph{reduced}\/ 
as a scheme if the ideals $Q_i$ are radical for $1\leq i\leq p$.
We also say that $\V(\Ann(M))$ is \emph{weakly reduced}\/ 
as a scheme if the ideals $Q_i$ are radical for $1\leq i\leq k$ 
and if $\fq_i=\fm$ for $k+1\leq i\leq p$, where 
$\fm=\langle x_1, \dots, x_n\rangle$ is the maximal 
ideal of $S$ at $0$; in other words, the only possible 
embedded component is at $0$.   

Suppose now that $A$ is a connected, locally finite, graded, 
graded-com\-mutative $\C$-algebra defined over $\Q$.
As shown in \cite[Proposition 6.2]{Papadima-Suciu15}, 
there is then a commutator-relators group $G$ such that 
the algebras $H^*(G,\C)$ and $A$ have the same quadratic 
closure, and hence have the same first resonance variety. 

On the other hand, as proved in \cite[Theorem 3.9]{Matei-Suciu00} 
(see also \cite{Suciu11, Cohen-Schenck15}), if $G$ is a 
commutator-relators group, then $\cR_1(G) = \V(\Ann(\fB(G)))$, 
where recall $\fB(G):=\fB(\fh(H^*(G,\C)))$ is the infinitesimal 
Alexander invariant of $G$. Thus, the first resonance variety 
of the algebra $A$ can be written as 
\begin{equation}
\label{eq:r1a}
\cR_1(A) = \V(\Ann(\fB(A))).  
\end{equation}
Thus, it is natural to view $\cR_1(A)$ as the set of closed points 
in the subscheme of $\spec(S)$ defined by $\Ann\big(\fB(A))$, 
which we call the {\em resonance scheme}\/ of $A$.  
Moreover, the resonance scheme of $A$ depends only on the quadratic 
closure $\bar{A}$ defined in \eqref{eq:quadclose}, 
that is,  $\Ann\big(\fB(A)=\Ann\big(\fB(\bar{A}))$.

More generally, for each $d\ge 1$, 
the depth $d$ resonance variety $\cR_d(G)$ 
can be viewed as the support variety of the annihilator of 
the $d$-th exterior power of the $S$-module $\fB(G)$,
\begin{equation}
\label{eq:depthdReso}
\cR_d(G)= \bV\, \bigg(\Ann\Big(\bigwedge^d\fB(G)\Big)\bigg)\, , 
\end{equation}
see \cite{Matei-Suciu00, Papadima-Suciu06}. Hence, we may define 
the depth $d$ resonance scheme of the graded algebra $A$ as the 
scheme defined by the associated primes of the annihilator ideal of 
$\bigwedge^d\fB(G)$.

\subsection{Bounding the resonance variety}
\label{subsec:bound}
The next lemma provides a `lower-bound' for the ideal $\Ann(\fB(A))$ and
an `upper bound' for the variety $\cR_1(A)$, in the case when the infinitesimal 
Alexander invariant of $A$ admits a suitable presentation.  

Recall that for 
a finitely generated $S$-module $M$, we let $E_0(M)$ be the ideal of 
maximal minors for a presentation matrix for $M$. As is well known,
this ideal does not depend on the choice of  presentation
$S^m\rightarrow S^n \rightarrow M$; furthermore,   
given such a presentation,  
$\Ann(M)^n \subseteq E_0(M) \subseteq \Ann(M)$.
In particular, if $M$ is a cyclic $S$-module 
(i.e., it is generated by a single element), then $E_0(M)=\Ann(M)$.

\begin{lemma} 
\label{lem:bound lemma}
Let  $\fB=\fB(A)$ be the infinitesimal Alexander invariant of 
a graded algebra $A$ as above. 
Suppose $\fB$ admits a block-triangular presentation matrix $\Omega$, 
with diagonal blocks $\Omega_{ii}$ for $1\leq i\leq q$.  
Let $\fB_{i}$ denote the $S$-module with presentation matrix  $\Omega_{ii}$. 
Then
\begin{enumerate}
\item\label{it1} 
 $E_0(\fB) \supseteq  \prod_{i=1}^q  E_0(\fB_{i})$.
\item\label{it2}
 $\cR_1(A)\subseteq  \bigcup_{i=1}^q \V(\Ann(\fB_{i}))$.
\item\label{it3}
 Furthermore, if each $\fB_i$ is a cyclic module, then
$\Ann(\fB) \supseteq  \prod_{i=1}^q  \Ann(\fB_{i})$.
\end{enumerate}

 \end{lemma}
 
\begin{proof}
The first claim follows from the standard way of computing determinants 
of block-triangular matrices. 
The second claim follows at once from statement \eqref{it1} and equation \eqref{eq:r1a}. 
The last claim follows from statement \eqref{it1}  and the paragraph preceding the lemma.   
\end{proof} 

\subsection{Chen ranks and resonance varieties}
\label{subsec:chen-resonance}

Recently, D.~Cohen and H.~Schenck proved the following theorem,
which establishes the Chen ranks conjecture from \cite{Suciu01} in 
a wider setting.  

A subspace $L\subseteq H^1(G;\C)$ is said to be \emph{$p$-isotropic}\/ (for some $p\ge 0$) 
if the restriction of the cup product map $H^1(G;\C)\wedge H^1(G;\C)\to H^2(G;\C)$ to 
$L\wedge L$ has rank $p$. In particular,  a subspace $L\subseteq H^1(G;\C)$  is 
$0$-isotropic (or simply, isotropic)
if the restriction of the cup product map $L\wedge L\to H^2(G;\C)$ is trivial.
Finally, two subspaces $U$ and $V$ of $H^1(G;\C)$
are said to be {\em projectively disjoint}\/ if $U\cap V=\{0\}$.  
  
\begin{theorem}[\cite{Cohen-Schenck15}]
\label{thm:Chenrankformula}
Let $G$ be a finitely presented, commutator-relators $1$-formal group.
Assume that the components of $\cR_1(G)$ are $0$-isotropic, 
projectively disjoint, and weakly reduced as a scheme. 
Then, for all $k \gg 0$, the Chen ranks of $G$ are given by
\begin{equation}
\label{eq:ChenranksConjecture}
\theta_k(G)= \sum_{m\geq 2}h_m(G) \cdot \theta_k(F_m), 
\end{equation}
where $h_m(G)$ is the number of $m$-dimensional components of $\cR_1(G)$.
\end{theorem}

In the same paper, Cohen and Schenck showed that 
the first resonance varieties of the McCool groups satisfy 
the hypotheses of Theorem \ref{thm:Chenrankformula}, 
and that the Chen ranks of these groups are given by
\begin{equation}
\label{eq:chen ps}
\theta_k(P\Sigma_n)=(k-1)\binom{n}{2}+(k^2-1)\binom{n}{3}, \quad \text{for $k\gg 0$}.
\end{equation}
 
In the sections that follow, we will compute the first resonance variety
$\cR_1(P\Sigma^+_n)$ and its scheme structure. Rather surprisingly, 
the Chen ranks formula \eqref{eq:ChenranksConjecture} does not hold for 
the groups $P\Sigma^+_n$ with $n\geq 4$. We will show that not all the 
components of $\cR_1(P\Sigma^+_n)$ are isotropic, and that $\cR_1(P\Sigma^+_n)$
is not weakly reduced as a scheme, as soon as  $n\geq 4$. 
Thus, in this range, the upper McCool groups $P\Sigma^+_n$ do not satisfy 
all the hypothesis of Theorem \ref{thm:Chenrankformula}.  

\section{The first resonance variety of $P\Sigma_n^+$}
\label{sec:resonance}

In this section, we compute the first resonance varieties of the 
upper McCool groups and apply the results to analyze several
properties of these groups.

\subsection{The first resonance variety of $P\Sigma_n^+$}
\label{subsec:resupper}

We are now ready to describe the first resonance variety 
of the upper McCool group $P\Sigma_n^+$, for all $n\ge 2$. 
Throughout, we 
will identify $H^1(P\Sigma_n^+;\C)$ with the $\C$-vector space 
$\C^{\binom{n}{2}}$,  endowed with the basis  $\{u_{ij} \mid 1\leq j<i \leq n\}$ 
provided by Theorem \ref{thm:Cohen-P-V-Wu}.  As before, $x_{ij}$ will 
denote the dual coordinate functions. 
For $n=2$, we have that  $\cR_1(P\Sigma_n^+)=\cR_1(\Z)=\{0\}$.

\begin{theorem}
\label{thm:Resonance}
For each $n\ge 3$, the resonance variety $\cR_1(P\Sigma_n^+)$ decomposes 
into irreducible components as 
\begin{equation}
\label{eq:r1ps}
\cR_1(P\Sigma_n^+)=\bigcup\limits_{2\leq j<i\leq n} L_{ij}, 
\end{equation}
where $L_{ij}\cong \C^{j}$ is the linear subspace of $\C^{\binom{n}{2}}$
defined by the equations 
\begin{equation}
\label{eq:L_ij}
\left\{
\begin{array}{ll}
x_{i,l}+x_{j,l}=0, & \text{for $1\leq l\leq j-1$};\\[2pt]
x_{i,l}=0 & \text{for $j+1\leq l\leq i-1$};\\[2pt]
x_{s,t}=0 & \text{for $s\neq i, s\neq j$, and $1\leq t<s$}.
\end{array}
\right.
\end{equation}
\end{theorem}
 
\begin{proof}
Fix $n\ge 3$, and write 
$L=\bigcup\limits_{2\leq j<i\leq n} L_{ij}$. 
We claim that $L= \cR_1(P\Sigma_n^+)$.
In order to verify the forward inclusion, 
we need to check that $L_{ij}\subseteq \cR_1(P\Sigma_n^+)$ for all $i>j$.
If $a\in L_{ij}$ is non-zero, then  the system of linear equations 
\eqref{eq:L_ij} implies that $a$ is of the form 
\begin{equation}
\label{eq:aformula}
a=\sum_{l=1}^{j-1}a_{il}(u_{il}-u_{jl})+a_{ij}u_{ij}.
\end{equation} 
Using Theorem \ref{thm:Cohen-P-V-Wu}, it is easy to check that $a\cdot u_{ij}=0$.
Hence, from \eqref{eq:FirstResonance},
we obtain that $L_{ij}\subseteq \cR_1(P\Sigma_n^+)$.
 
For the reverse inclusion, we use the Gr\"{o}bner basis of the infinitesimal 
Alexander invariant $\fB_n$ provided by Theorem \ref{thm:GrobnerBasis}. 
For each diagonal vector $\vec{w}_{ijk}$ from Corollary \ref{cor:diag},
the equation $\vec{w}_{ijk}=0$  defines a linear space $L_{ijk}$: the linear 
entries from $\vec{v}_{ijk}$ yield equations of the form $x_{il}+x_{jl}+x_{kl} =0$ 
for $1\leq l\leq k-1$, 
$x_{ik}+x_{jk}=0$  for $1\leq k<j<i\leq n$, 
and $x_{st} =0$  for  $\{s,t\}\not\subset \{i,j,k,l\}$ and $1\leq l\leq k-1$, 
while the quadratic entries of $\vec{w}_{ijk}$ yield equations of the form 
$x_{ks}=0$ for $1\leq s\leq k-1$.

Clearly,  $L_{ijk}$ is a subspace of the linear space $L_{i,j,j-1}=L_{ij}$ 
defined by equations \eqref{eq:L_ij}. By Lemma \ref{lem:bound lemma},
we have that $\cR_1(P\Sigma_n^+)\subseteq L$, and this establishes 
the claim that equality \eqref{eq:r1ps} holds.

Finally, it is also clear that each linear subspace $L_{ij}$ ($2\le j<i\le n$) is an irreducible 
variety, and no $L_{ij}$ is properly included is some distinct $L_{kl}$. This  
shows that \eqref{eq:r1ps} is indeed the irreducible decomposition of 
$\cR_1(P\Sigma_n^+)$, thereby completing the proof.  
\end{proof}

\subsection{Isotropicity}
\label{subsec:isotrope}

The next theorem lists some of the basic properties of the 
(first) resonance varieties of the upper McCool groups.   

\begin{theorem}
\label{thm:resonance}
Let $L_{ij}$ ($2\leq j<i\leq n$) be the irreducible components of $\cR_1(P\Sigma_n^+)$, 
the first resonance variety of the upper McCool group $P\Sigma_n^+$.  Then:
 \begin{enumerate}
  \item  \label{rc1}
Each $L_{ij}$ is a linear subspace of dimension $j$, 
 with basis $\{u_{jl}-u_{il}, u_{ij} \mid 1\leq l\leq j-1\}$.
 \item  \label{rc2}
 $L_{ij}\cap L_{st}=\{0\}$ if $(i,j)\neq (s,t)$.
 \item \label{rc3}
The subspace $L_{ij}$ is $0$-isotropic for $j=2$ and $\binom{j-1}{2}$-isotropic $j\geq 3$. 
 \item \label{rc4}
 $\cR_1(P\Sigma_n^+)=\cR_1(P\Sigma_{n+1}^+)\cap H^1(P\Sigma_n^+;\C)$.
 \end{enumerate}
\end{theorem}

\begin{proof}
\eqref{rc1} It follows from \eqref{eq:aformula} that $L_{ij}$ is the linear 
subspace of  $\C^{\binom{n}{2}}$ with the specified basis.

\eqref{rc2} 
Using the defining equations \eqref{eq:L_ij} for the subspaces  $L_{ij}$ and $L_{st}$, 
it is readily seen that these two subspaces intersect only at $\{0\}$.
 
\eqref{rc3} 
Consider a subspace $L_{ij}$ as in \eqref{rc1}. 
From Theorem \ref{thm:Cohen-P-V-Wu}, we know that $(u_{jl}-u_{il})u_{ij}=0$ and
$(u_{jl}-u_{il})(u_{jk}-u_{ik})\neq 0$ for $1\leq l< k\leq j-1$.
If $j=2$, the subspace $L_{i2}$ has basis $\{u_{21}-u_{i1}, u_{i2}\}$; hence, 
it is $0$-isotropic. 
If $j\geq 3$, the image of the cup product map $L_{ij}\wedge L_{ij}\to H^2(P\Sigma_n^+;\C)$ 
is a linear subspace with basis 
$\{(u_{jl}-u_{il})(u_{jk}-u_{ik})  \mid1\leq l< k\leq j-1 \}$.
Hence, $L_{ij}$ is $\binom{j-1}{2}$-isotropic.

\eqref{rc4}  By Theorem \ref{thm:Cohen-P-V-Wu}, we can construct a basis for 
$H^1(P\Sigma_{n+1}^+;\C)$ by taking the union of a basis of 
$H^1(P\Sigma_n^+;\C)$ with the set $\{u_{n+1,1}, \dots, u_{n+1,n}\}$.
By Theorem \ref{thm:Resonance}, we have that 
\[
\cR_1(P\Sigma_n^+)=\bigcup\limits_{2\leq j<i\leq n} L_{ij}
\:\text{ and }\: \cR_1(P\Sigma_{n+1}^+)=\bigcup\limits_{2\leq j<i\leq n+1} V_{ij},
\]
where $L_{ij}=V_{ij}\cap H^1(P\Sigma_n^+;\C)$ for $2\leq j<i\leq n$, and
$V_{n+1,j}\cap H^1(P\Sigma_n^+;\C)=\{0\}$ for $2\leq j \leq n$.  
The claim follows.
\end{proof}

\subsection{Split monomorphisms}
\label{subsec:split}
For each $n\ge 1$, there is a split injection 
$P\Sigma_n^+\rightarrow P\Sigma_{n+1}^+$.
Furthermore, the inclusion  $\iota\colon P\Sigma_n^+\inj P\Sigma_n$ is a
split monomorphism for $n=3$.  However, using the first resonance varieties, 
we can rule out the existence of a splitting homomorphism for $\iota$ when 
$n\ge 4$.   We start by recalling a result of Cohen \cite{DCohen09} and 
Cohen--Schenck \cite{Cohen-Schenck15}, based on the computation 
of the cohomology ring of $P\Sigma_n$ by Jensen--McCammond--Meier 
from \cite{Jensen-McCammond-Meier06}.

\begin{theorem}[\cite{DCohen09, Cohen-Schenck15}]
\label{thm:res mccool} 
For each $n\ge 2$, the first resonance variety of the  group $P\Sigma_n$
decomposes into irreducible components as
\[
 \cR_1(P\Sigma_n)=\bigcup_{1\leq i<j\leq n} C_{ij} \cup \bigcup_{1\leq i<j<k\leq n} C_{ijk}, 
\]
 where $C_{ij}$ is the plane defined by the equations $x_{pq}=0$ for $\{p,q\}\neq \{i,j\}$ and 
 $C_{ijk}$ is the $3$-dimensional linear subspace defined by the equations 
 $x_{ij}+x_{kj}=x_{ji}+x_{ki}=x_{ik}+x_{jk}=0$  
 and $x_{st}=0$ for $\{s,t\}\nsubseteq \{i,j,k\}$.  
 Furthermore, all these components are isotropic.
\end{theorem}

We can now answer a question raised by Paolo Bellingeri. 

\begin{prop}
\label{prop:split}
There is no epimorphism from $P\Sigma_n$ to $P\Sigma^+_n$ for $n\geq 4$. In particular,
the inclusion $\iota\colon P\Sigma_n^+\to P\Sigma_n$ admits no splitting for $n\geq 4$. 
\end{prop}

\begin{proof}
Suppose  $\sigma\colon P\Sigma_n \surj P\Sigma_n^+$ is an epimorphism. 
By Lemma \ref{lem:PS06}, the epimorphism $\sigma$ induces a monomorphism 
$\sigma^{\ast}\colon H^1(P\Sigma_n^+;\C)\inj H^1(P\Sigma_n;\C)$ 
which takes $\cR_1(P\Sigma_n^+)$ to $\cR_1(P\Sigma_n)$.

Now, we know from Theorem \ref{thm:res mccool} that   
$\cR_1(P\Sigma_n)$ is a union of linear spaces of dimension $2$ or $3$. 
On the other hand, Theorem \ref{thm:resonance} insures that 
$\cR_1(P\Sigma_n^+)$ has irreducible components  
which are linear spaces of dimension $n-1$. Hence, for $n\geq 5$, 
there is no epimorphism from $P\Sigma_n$ to $P\Sigma_n^+$.

For $n=4$, Theorem \ref{thm:resonance} also tells us that the 
irreducible component $L_{43}\subset \cR_1(P\Sigma_4^+)$ 
is not isotropic. For any $a,b \in L_{43}$ such that $a\cup b\neq 0$,
we have that $\sigma^*(a)\cup \sigma^*(b)=\sigma^*(a\cup b)\neq 0$, 
by the injectivity of $\sigma^*$.  Hence, $\sigma^*$ 
must take the non-isotropic component $L_{43}\subset \cR_1(P\Sigma_4^+)$ 
to a non-isotropic component of $\cR_1(P\Sigma_4)$. 
 However,  all 
irreducible components of $\cR_1(P\Sigma_4)$ are isotropic subspaces. 
This is a contradiction, and so we are done.
\end{proof}

\begin{remark} 
\label{rem:split3}
The canonical inclusion $\iota\colon P\Sigma_3^+\inj P\Sigma_3$ 
\emph{does}\/ admit a splitting, for instance, the homomorphism $\sigma\colon P\Sigma_3\inj P\Sigma_3^+$
defined by sending $\{\alpha_{21},\alpha_{31},\alpha_{32}, \alpha_{12},\alpha_{13},\alpha_{23}\}$
to $\{\alpha_{21},\alpha_{31},\alpha_{32}, \alpha_{32}^{-1}$, $\alpha_{21}^{-1},\alpha_{31}^{-1}\}$, 
respectively.  The induced homomorphism in first cohomology, 
$\sigma^*\colon H^1(P\Sigma_3^+,\Z) \to H^1(P\Sigma_3,\Z)$,  
sends $\{u_{21}, u_{31}, u_{32}\}$ to $\{u_{21}-u_{13}, u_{31}-u_{23}, u_{32}-u_{12}\}$, 
respectively; consequently, $\sigma^*$ takes $\cR_1(P\Sigma_3^+)$ to the linear subspace 
$C_{123}\subset \cR_1(P\Sigma_3)$.
\end{remark}

\subsection{Quasi-projectivity}
\label{subsec:qp}
A finitely presented group $G$ is said to be a \emph{quasi-projective group}\/ 
if it can be realized as $G=\pi_1(M)$, where $M$ is a smooth, connected, 
complex quasi-projective variety.   In 1958, J.-P. Serre asked the following 
question:  Which finitely presented groups are quasi-projective?
Combining Theorem B from \cite{Dimca-Papadima-Suciu09} with Theorem 4.2
from \cite{Dimca-Papadima-Suciu08}, we have the following obstruction for 
quasi-projectivity of a $1$-formal group.

\begin{theorem}
\label{thm:quasiproj}
Let $G$ be a quasi-projective, $1$-formal group.   Then each positive-dimensional 
irreducible component of the first resonance variety $\cR_1(G)$ is a linear subspace 
of $H^1(G;\C)$ which is either $0$-isotropic and of dimension at least $2$, or $1$-isotropic 
and of dimension of at least $4$.
\end{theorem}

For instance, the pure braid groups $P_n$ are both quasi-projective and $1$-formal, 
and all the components of $\cR_1(P_n)$ are $0$-isotropic, $2$-dimensional subspaces, 
for $n\ge 3$.   On the other hand, as an application of this theorem and our own results, 
we obtain the following corollary.

\begin{prop}
\label{prop:quasiproj}
For each $n\geq 4$, the upper McCool groups $P\Sigma_n^+$ 
is not quasi-projective.
\end{prop}

\begin{proof}
Since  $n\geq 4$, Theorem \ref{thm:resonance} implies that 
$\cR_1(P\Sigma_n^+)$ contains a component $L_{43}$ which 
is a $3$-dimensional, $1$-isotropic linear subspace of $H^1 (P\Sigma_n^+)$.  
On the other hand, by Theorem \ref{thm:McCoolFormal}, all the upper 
McCool groups $P\Sigma_n^+$ are $1$-formal.
Hence, by Theorem \ref{thm:quasiproj}, the group $P\Sigma_n^+$ is 
not quasi-projective.
\end{proof}

\begin{remark}
\label{rem:mcqp}
The same yoga as in the previous corollary cannot be applied to the full
McCool groups $P\Sigma_n$, since, as we saw in Theorem \ref{thm:res mccool},  
the components of $\cR_1(P\Sigma_n)$ are isotropic and of dimension $2$ and $3$.  
To the best of the authors' knowledge, it is unknown whether or not the groups $P\Sigma_n$ are 
quasi-projective for $n\ge 3$.
\end{remark}

\begin{remark}
\label{rem:pdiff}
Comparing the resonance varieties of $P\Sigma_n^+$ with those of $P_n$ 
and $\Pi_n$ yields  another proof of Corollary  \ref{cor:chen ppp}.  Indeed, 
for $n\ge 4$, all irreducible components of $\cR_1(P_n)$ and $\cR_1(\Pi_n)$ 
are  isotropic linear subspaces (of dimension $2$, respectively, $2,\dots, n-1$), 
whereas  $\cR_1(P\Sigma^+_n)$ has non-isotropic components.
\end{remark}

\section{The scheme structure of $\cR_1(P\Sigma_n^+)$}
\label{sec:scheme}

In this last section we determine the scheme structure defined by the 
ideal $\Ann(\fB(P\Sigma_n^+))$ on the 
resonance variety $\cR_1(P\Sigma_n^+)$.

\subsection{Two $S$-modules and their Hilbert series} 
\label{subsec:hsmod}
We start with some preparation. 
Let $\x=\{x_{ij}\mid 1\leq j<i\leq n\}$ be the dual of the standard basis 
of $H^1(P\Sigma_n^+;\C)$, and let $S=\C[\x]$ be the polynomial ring 
in those variables. Recall from Proposition \ref{prop:reducedPres} that the infinitesimal 
Alexander invariant $\fB_n=\fB(P\Sigma_n^+)$ has a presentation given by 
$\fB_n=S^{\binom{n}{3}}/ \im(\Psi)$, where $\im(\Psi)$ is the submodule of 
$S^{\binom{n}{3}}$ generated by the set $\cB=\bigcup_{1\leq k<j<i\leq n} \cB_{ijk}$ 
from Lemma \ref{lem:basisIJK}. 

Let $\fB_n^{\prime}$ be the quotient of  the $S$-module 
$\fB_n$ by the submodule generated by the set of monomials 
$\cE_{ijk}:=\{\ff:= x_{kp}\cdot r_{ijk} \mid  1\le p\leq  k-1 \}$.
Then $\fB_n^{\prime}$ has a presentation 
\begin{equation}
\label{eq:fbnprime}
\fB_n^{\prime}=S^{\binom{n}{3}}/ \fI,
\end{equation}
where $\fI$ is the submodule of $S^{\binom{n}{3}}$ generated by
the set $\cB^{\prime}=\bigcup_{1\leq k<j<i\leq n} \cB^{\prime}_{ijk}$ and 
\begin{equation}
\label{eq:cbnprime}
\cB^{\prime}_{ijk}=\cB_{ijk}\cup \cE_{ijk}.
\end{equation}

\begin{prop}
\label{prop:bprime-gb}
The set $\cB^{\prime}$ forms a Gr\"{o}bner basis for the submodule $\fI$.
\end{prop}
\begin{proof}
Comparing the set $\cD_{ijk}$ from \eqref{eq:dijk} with the set 
$\cE_{ijk}$, we see that each element in $\cD_{ijk}$ is of the form 
$x_{kl}\ff$ or $x_{jq}\ff$, for some $\ff\in \cE_{ijk}$.
In view of step 2 from the proof of Theorem \ref{thm:GrobnerBasis} 
given in Appendix \ref{sec:Appendix}, 
in order to reach the desired conclusion, we only need to 
check the vanishing of the $\fS$-polynomials $\fS(g,f)$ for all
$f\in \cE_{ijk}$ and $g\in \cG_{\bg}$ from \eqref{eq:b2}.  We have:
\begin{align*}
\fS(\bg_1,\ff)&= x_{kp}(-x_{jk}-x_{l_2k})\cdot r_{ijl_2} =-(x_{jk}+x_{l_2k})\bh_{8}^{(ijl_2)},\\
\fS(\bg_2,\ff)&= x_{kp}x_{jk}\cdot r_{ijl_2}=x_{jk} \bh_{8}^{(ijl_2)},\\ 
\fS(\bg_3,\ff)&=  -x_{kp}x_{jk}\cdot r_{il_3j}=-x_{jk}\bh_{8}^{(il_3j)},\\ 
\fS(\bg_4,\ff)&= x_{kp}x_{jk}\cdot r_{l_4ij}=x_{jk}\bh_{8}^{(l_4ij)}, 
\end{align*} 
where $\bh_8^{(***)}\in \cB_{***}$ is the corresponding element from 
Lemma \ref{lem:basisIJK}.  Hence, all these $\fS$-polynomials vanish, 
and we are done.
\end{proof}

Now choose a basis $\{e_{ijkl} \mid 1\leq l<k<j<i \leq n\}$ for the free module $S^{\binom{n}{4}}$, 
and let $J$ be the submodule of $S^{\binom{n}{4}}$ generated by the monomials 
$x_{st}e_{ijkl}$, where $1\leq t<s\leq n$, and  $(s,t)\neq (i,j)$.
We then define an $S$-module 
\begin{equation}
\label{eq:kn}
K_n:=S^{\binom{n}{4}}/J.
\end{equation}
For $3\leq j<i\leq n$, let $K_{ij}=S/I_{ij}$, where $I_{ij}$ is the ideal of $S$ generated by
the variables $x_{st}$  with $1\leq t<s\leq n$ for which $(s,t)\neq (i,j)$. 
The $S$-module $K_n$ can be decomposed as 
\begin{equation}
\label{eq:kn-bis}
K_n\cong \bigoplus_{3\leq j<i\leq n}\bigoplus_{1}^{\binom{j-1}{2}} K_{ij}\, .
\end{equation}
 
\begin{prop}
\label{prop:HilbBnprime}
For each $n\ge 4$, the following equalities hold:
\[ 
\Hilb(K_n,t)=\binom{n}{4}\frac{1}{1-t} \quad \textrm{ and } \quad
\Hilb(\fB_n^{\prime},t)
=\sum\limits_{s=2}^{n-1}\binom{s}{2}\frac{1}{(1-t)^{n-s+1}}\, .
\]
\end{prop}
\begin{proof}
The $S$-module $K_n$ decomposes as the direct 
sum of $\binom{j-1}{2}$ copies of sub-modules 
$K_{ij}=S/I_{ij}$ for $3\leq j<i\leq n$, where $I_{ij}$ is 
the ideal generated by the variables $x_{st}$ with $1\leq t<s\leq n$ and $(s,t)\neq (i,j)$.
Since $\Hilb(K_{ij},t)=1/(1-t)$, the first equality readily follows.

To prove the second equality, recall first that $\fB_n^{\prime}=S^{\binom{n}{3}}/ \fI$, 
with $ \fI$ the ideal with Gr\"{o}bner basis 
$\cB^{\prime}=\bigcup_{1\leq k<j<i\leq n} \cB^{\prime}_{ijk}$, 
where the set $\cB^{\prime}_{ijk}$ is given in \eqref{eq:cbnprime}. 
It is readily seen that 
\begin{equation}
\ini_{\succ}(\cB_{ijk}^{\prime})=
\left\{ 
x_{kl}\cdot r_{ijk}, ~
x_{ik}\cdot r_{ijk},   ~  x_{il}\cdot r_{ijk},  ~
x_{ab}\cdot r_{ijk},   \{a,b\}\not\subset \{i,j,k,l\}, 1\leq l\leq k-1
\right\}. 
\end{equation}
The claimed expression for the Hilbert series of $\fB_n^{\prime}$ 
now follows as in the proof of Theorem \ref{thm:HilbertSeries}.
\end{proof}

\subsection{A short exact sequence of $S$-modules} 
\label{subsec:sssmod}

In order to understand the annihilator ideal of the infinitesimal Alexander 
invariant $\fB_n$, we approximate it by a simpler quotient module, $\fB_n'$, 
and then study the kernel of the projection map, $K_n$.

\begin{theorem}
\label{thm:SES}
For each $n\geq 4$, there is a short exact sequence of graded $S$-modules, 
\[
\xymatrix{0\ar[r] & K_n \ar[r] & \fB_n \ar[r]^{p} & \fB_n^{\prime} \ar[r] & 0},
\]

\end{theorem}
\begin{proof}
From the definition of the modules $\fB_n^{\prime}$ and $\fB_n$, there 
is a canonical projection $p\colon \fB_n \surj \fB'_n$. 
Let us verify the claim that $\ker (p)=K_n$. Consider the sequence 
\begin{equation}
\xymatrix{
S^{\binom{n}{4}} \ar[r]^{\phi} &\fB_n \ar[r]^{p} & \fB_n^{\prime}\ar[r] & 0
}.
\end{equation}

Choose a basis $\{e_{ijkl} \mid 1\leq l<k<j<i \leq n\}$ for the 
free module $S^{\binom{n}{4}}$, and let the morphism $\phi$ be 
defined by $\phi(e_{ijkl})=x_{kl}r_{ijk}$. We then have 
$\ker(p)=\im(\phi)$, so the above sequence is exact in the middle. 
Hence we have a short exact sequence of $S$-modules, 
\begin{equation}
\label{eq:SESb}
\xymatrix{
0\ar[r] &S^{\binom{n}{4}}/\ker(\phi) \ar[r]^(.63){\bar\phi} 
&\fB_n \ar[r]^{p} & \fB_n^{\prime}\ar[r] & 0
}.
\end{equation}
In view of the Hilbert series computations from Proposition \ref{prop:HilbBnprime}
and Theorem \ref{thm:HilbertSeries}, we infer that $\Hilb(S^{\binom{n}{4}}/\ker(\phi),t)=\Hilb(K_n,t)$.

Using the Gr\"{o}bner basis $\cG$ for $\im(\Psi)$ from 
Theorem \ref{thm:GrobnerBasis}, it is easy to check that $\phi(J)$ is included in the 
submodule of $S^{\binom{n}{3}}$ generated by $\cG$. Hence, 
we have that $J\subseteq \ker(\phi)$ and there is a canonical surjection 
$K_n\surj S^{\binom{n}{4}}/\ker(\phi)$. Since both $S$-modules have the 
same Hilbert series, we conclude that $K_n\cong S^{\binom{n}{4}}/\ker(\phi)$. 
This completes the proof.
\end{proof}
 
The next following proposition details the relationship between 
the supports of the $S$-modules $K_n$, $\fB_n$, and $\fB_n^{\prime}$.

\begin{prop}
\label{prop:Bnprime}
For each $n\ge 4$,  we have that 
$ \V(\Ann(K_n))\subseteq \V(\Ann (\fB_n))=\V(\Ann(\fB_n^{\prime}))$.
\end{prop}
\begin{proof}
Let us start by noting that 
\begin{equation}
\label{eq:AnnKn}
\Ann(K_n)=\Ann\bigg(\bigoplus_{3\leq j<i\leq n}\bigoplus_{1}^{\binom{j-1}{2}} K_{ij}\bigg)=
\bigcap_{3\leq j<i\leq n} \Ann(K_{kl})=\bigcap_{3\leq j<i\leq n} I_{ij}\, .  
\end{equation}
Hence, $\bV(\Ann(K_n))$ is a union of lines $\bV(I_{ij})$ defined by equations
$x_{st}=0$ for $1\leq t<s\leq n$ and $(s,t)\neq (i,j)$.

Now let $\fB_n^{\prime}(ijk)$ be the quotient of $\fB_n^{\prime}$
by the ideal generated by $\{r_{stl}\mid (s,t,l)\neq (i,j,k)\}$. We then have
$\bV(\Ann(\fB_n^{\prime}(ijk)))\subseteq \bV(\Ann(\fB_n^{\prime}))$. 
With the help of Lemma \ref{lem:PhiFormula}, direct computation shows that the variety
$\bV(\Ann(\fB_n^{\prime}(ijk)))$ is the $2$-plane $P_{ijk}$ defined by the equations 
$x_{ik}+x_{jk}=0$ and $x_{st}=0$ for $\{s,t\}\notin \{i,j, k\}$. 

The short exact sequence from Theorem \ref{thm:SES} implies 
that 
\begin{equation}
\label{eq:annshort}
\V(\Ann (\fB_n))=\V(\Ann(\fB_n^{\prime}))\cup \V(\Ann(K_n)).
\end{equation}
Using \eqref{eq:AnnKn}, we see that  $\V(\Ann(K_n))\subseteq \V(\Ann(\fB_n))$ and 
\[
\V(\Ann(K_n))\subseteq \bigcup_{1\leq k<j<i\leq n}\bV(\Ann(\fB_n^{\prime}(ijk)))
\subseteq\V(\Ann(\fB_n^{\prime})).
\] 
Therefore, $\V(\Ann (\fB_n))=\V(\Ann(\fB_n^{\prime}))$, 
thereby completing the proof.
\end{proof}

\subsection{Resonance scheme structure}
\label{subsec:res-scheme}
We now analyze the scheme structure of the annihilator ideals 
of the modules $K_n$ and $\fB_n^{\prime}$ defined above. 

\begin{theorem}
\label{thm:reducedscheme}
The resonance schemes defined by $\Ann(K_n)$ and $\Ann(\fB_n^{\prime})$ are reduced.
\end{theorem}
\begin{proof}
From \eqref{eq:AnnKn}, it is clear that $\Ann(K_n)$ is reduced.   
So we are left with proving the second assertion. 
Let $Q_{ij}$ be the ideal generated by the linear forms 
$x_{il}+x_{jl}$, $x_{ir}$, and $x_{st}$, where $1\leq l\leq j-1$, 
$j+1\leq r\leq i-1$, $s\neq i$, and $s\neq j, 1\leq t<s$. Clearly, each 
$Q_{ij}$ is a prime ideal. From \eqref{eq:r1a} and Theorem \ref{thm:Resonance}, 
we infer that the set of minimal primes of $\fB_n$ is $\{Q_{ij} \mid 2\leq j<i\leq n\}$.

By Proposition \ref{prop:Bnprime}, we have that $\V(\Ann (\fB_n))=\V(\Ann(\fB_n^{\prime}))$.
Therefore, the set of minimal primes of $\fB_n^{\prime}$ coincides with the set of minimal 
primes of $\fB_n$, and so 
\begin{equation}
\label{eq:AnnBnsubset}
\Ann(\fB_n^{\prime})\subseteq \bigcap_{2\leq j<i\leq n} Q_{ij}.
\end{equation}
 
Applying Lemma \ref{lem:bound lemma} to the $S$-module $\fB_n^{\prime}$, 
we infer that the product of the ideals $Q_{ij}$ from above is contained in 
$\Ann(\fB_n^{\prime})$.  Since the sum of those ideals is $S$, we conclude that 
\begin{equation}
\label{eq:AnnBnprime}
\bigcap\limits_{2\leq j<i\leq n} Q_{ij}=\prod\limits_{2\leq j<i\leq n} Q_{ij}\subseteq 
\Ann(\fB_n^{\prime}).
\end{equation}
Hence, the annihilator of $\fB_n^{\prime}$ 
has primary decomposition 
\begin{equation}
\label{eq:AnnBn}
\Ann(\fB_n^{\prime})=\bigcap\limits_{2\leq j<i\leq n} Q_{ij},
\end{equation}
with each $Q_{ij}$ a prime ideal.
This completes the claim that $\Ann(\fB_n^{\prime})$ is reduced. 
\end{proof}

We are now ready to describe the scheme structure of the first resonance 
variety $\cR_1(P\Sigma_n^+)$.

\begin{theorem}
\label{thm:resonance scheme}
The resonance scheme of the upper McCool group $P\Sigma_n^+$ 
defined by the ideal $\Ann (\fB_n)$ consists 
of the isolated components $L_{ij}$ with $2\le j<i\le n$ listed in 
Theorem \ref{thm:resonance}, together with $1$-dimensional, 
embedded components $L_{ij}' \subset L_{ij}$ defined by the equations
$x_{st}=0$ for $1\leq t<s\leq n$ and $(s,t)\neq (i,j)$, for all $3\le j<i\le n$.
\end{theorem}
\begin{proof}
Recall from Theorem \ref{thm:SES} that we have a short exact sequence 
$0\to K_n \to \fB_n \to \fB_n^{\prime}\to  0$. As a consequence, we have 
inclusions of sets of associated primes,
\begin{equation}
\label{eq:AnnKnbis}
\Ass (K_n)\subseteq \Ass (\fB_n)\subseteq \Ass(\fB_n^{\prime})\cup \Ass(K_n).
\end{equation}

On the other hand, Theorems \ref{thm:Resonance} and \ref{thm:reducedscheme}  
imply that $\Ass (\fB_n^{\prime})\subseteq \Ass (\fB_n)$.  Combining this inclusion with 
\eqref{eq:AnnKnbis}, we find that 
\begin{equation}
\label{eq:Annfbneq}
\Ass (\fB_n)= \Ass(\fB_n^{\prime})\cup \Ass(K_n).
\end{equation}

Hence, the isolated components of the resonance scheme are the 
varieties associated to the associated primes of $\fB'_n$, 
while the embedded components are the varieties associated to
the associated primes of $K_n$, i.e., the set of primes $I_{ij}$ 
(with duplicates removed). This completes the proof.
\end{proof}

As a quick application of this theorem, we obtain the following corollary.

\begin{corollary}
\label{cor:notreduced}
For each $n\geq 4$, 
the first resonance variety $\cR_1(P\Sigma_n^+)$ is not 
weakly reduced as a scheme.
\end{corollary}

\begin{example}
\label{ex:scheme4}
By Theorem \ref{thm:resonance scheme}, the resonance scheme of $P\Sigma_4^+$
contains the isolated components $L_{32}$, $L_{42}$ and $L_{43}$,
and the embedded component $L_{43}^{\prime}$. 
Moreover, from the presentation of $\fB_4=\fB(P\Sigma_4^+)$ given in \eqref{eq:PresentationPsi}, 
we find that a primary ideal corresponding 
to $L_{43}^{\prime}$ in the primary decomposition of  $\Ann(\fB_4)$ is 
\[
J_{43}=\ideal(x_{41}+x_{31}+x_{21},x_{31}x_{21}, x_{32}x_{21},x_{42}x_{21},
 x_{42}x_{31}+x_{32}x_{31},x_{42}x_{32},x_{21}^2 , x_{31}^2 ,x_{32}^2, x_{42}^2),
 \]
with radical ideal $\sqrt{J_{43}}=\ideal(x_{21},x_{31},x_{32},x_{41},x_{42})$.  
\end{example}

\subsection{Higher depth resonance}
\label{subsec:exdisc}
Recall from Theorem \ref{thm:resonance} that each isolated 
component of the scheme defined by $\Ann(\fB_n)$ is a linear subspace 
$L_{ij}$ spanned by the set $\{u_{jl}-u_{il}, u_{ij} \mid 1\leq l\leq j-1\}$.
By Theorem \ref{thm:resonance scheme}, if $j\ge 3$, this linear space contains an embedded component, 
which is the $1$-dimensional linear subspace $L_{ij}'$ spanned by the vector $u_{ij}$. 
The relationship between the isolated components and the embedded components 
of the resonance scheme $\Ann(\fB_n)$ can then be described as
\begin{equation}
\label{eq:lijprime}
L_{ij}'=\big\{a \in L_{ij} \mid \text{$a\cup b = 0$, for all $b\in L_{ij}$}\big\}.
\end{equation}
In other words, $L_{ij}'$ is the maximal subspace of $L_{ij}$ which is 
perpendicular to $L_{ij}$, with respect to the cup-product map on 
$H^1(P\Sigma_n^+,\C)$. 

As another application, we obtain some partial information on the 
higher-depth resonance varieties of the upper McCool groups 
$P\Sigma_n^+$. 

\begin{prop}
\label{prop:depthresonance}
For all $d\geq 2$, the following inclusion holds: 
\begin{equation}
\label{eq:highdepth}
\cR_d (P\Sigma_n^+)\supseteq\bigcup_{d+1\le  j<i\le n} L_{ij}^{\prime}\, .
\end{equation}
\end{prop}
\begin{proof}
Let  $W\subset H^1(P\Sigma_n^+,\C)$ be the $(j-1)$-dimensional 
linear subspace spanned by $\{u_{ik}-u_{jk} \mid 1\leq k<j<i\leq n\}$. 
By Theorem \ref{thm:Cohen-P-V-Wu}, we have that 
$u_{ij}(u_{ik}-u_{jk})=0$. Therefore, by  \eqref{eq:res1}, 
$u_{ij}\in \cR_d (P\Sigma_n^+)$ for $d\leq j-1$.  
Hence, $L_{ij}^{\prime}=\spn\{u_{ij}\}$ is included in $\cR_d (P\Sigma_n^+)$ 
for $d+1\le  j<i\le n$. 
\end{proof}

\begin{remark}
\label{rem:high-depth}
It seems reasonable to expect that the depth-$d$ resonance varieties 
of $P\Sigma_n^+$ have a similar decomposition into irreducible components 
as those in depth-$1$.   More precisely, we conjecture that inclusion \eqref{eq:highdepth} 
holds as equality for $d\geq 2$, and gives the decomposition into irreducible components 
of the resonance varieties $\cR_d (P\Sigma_n^+)$.  Furthermore, we expect 
that these varieties are reduced as schemes for all $d\ge 2$.  We have 
verified that this conjecture holds for $n\le 5$, as well as for $n=6$ and $d=2$.
\end{remark}

\section{Appendix: Proof of Theorem \ref{thm:GrobnerBasis}} 
\label{sec:Appendix}

Let $\Psi\colon S^m\to S^{\binom{n}{3}}$ be the $S$-linear map from 
Proposition \ref{prop:reducedPres}. We know from Lemma \ref{lem:basisIJK} 
that the $S$-module $\im(\Psi)$ is generated by the set 
$\cB=\bigcup_{1\leq k<j<i\leq n} \cB_{ijk}$, where $\cB_{ijk}$ consists 
of the elements from \eqref{eq:Rijk}.  Let 
$\cG=\bigcup_{1\leq k<j<i\leq n}  (\cB_{ijk}\cup \mathcal{D}_{ijk})$, where 
$\mathcal{D}_{ijk}$ is given in \eqref{eq:dijk}.
Our task is to show that the set $\cG$ is a Gr\"{o}bner basis for $\im(\Psi)$. 
We do this in two steps.

\subsubsection*{Step 1.} 
We first show that each set $\cD_{ijk}$ is included in $\im(\Psi)$.
Using the description of the sets $\cB_{ijl}$ and $\cB_{ikp}$ from Lemma \ref{lem:basisIJK}, 
we see that for $1\leq p\leq l<  k$ and $1\leq q\leq k$, the elements 
\begin{equation*}
\label{eq:gensGrobf}
\begin{cases}
\ff_1 :=(-{x}_{jl}-{x}_{kl})\cdot{r}_{ijk}+{x}_{jk}\cdot{r}_{ijl}\\
\ff_2 :={x}_{jl}\cdot{r}_{ijk}+{x}_{ik}\cdot{r}_{ijl}
\end{cases} 
\ff_3 :=-{x}_{kp}\cdot{r}_{ijk}+{x}_{ij}\cdot{r}_{ikp}  \quad
\begin{cases}
\ff_4:=x_{kp}\cdot r_{ijl} \\ 
\ff_5:=x_{jq}\cdot r_{ikp} \, .
\end{cases}
\end{equation*} 
are in $\cB\subset \im(\Psi)$.  Direct computation shows that
 \begin{equation}
\label{eq:gensGrob}
\begin{cases}
x_{kl}x_{kp}\cdot r_{ijk}&=  (x_{ik}+x_{jk})\ff_4- x_{kp}(\ff_1+\ff_2)\\
x_{jq}x_{kp}\cdot r_{ijk}&= x_{ij}\ff_5-x_{jp}\ff_3\, ,
\end{cases}
\end{equation} 
from which we conclude that indeed $\cD_{ijk}\subset \im(\Psi)$.

\subsubsection*{Step 2.} 
We now show that all $\fS$-polynomials between pairs of elements of $\cG$ vanish.
Clearly, $\fS$-polynomials of elements whose initial terms contain distinct basis 
elements of $(I^2)^*\otimes S$ vanish; thus, 
we only need to calculate the $\fS$-polynomials of pairs of elements from 
$\cG_{ijk}=\cB_{ijk} \cup \cD_{ijk}$,  for $1\leq k<j<i\leq n$. 
To start with, note that the subset
\begin{equation}
\label{eq:b3}
\cG_{\bh}:=\{ \bh_1, \bh_2, \bh_3, \bh_4, \bh_5, \bh_6, \bh_7, 
\bh_{8}, \bh_9, \bh_0 \}\subset \cG_{ijk}
\end{equation} 
only contains elements of the form $p\cdot r_{ijk}$, where $p\in S$.  
Thus, it is easy to check the vanishing of all $\fS$-polynomials of 
pairs of elements from this subset.  Next, we consider the subset 
\begin{equation}
\label{eq:b2}
\cG_{\bg}:=\left\{\bg_{1}, \bg_2, \bg_3,\bg_4 \right\}\subset \cG_{ijk},
\end{equation} 
and check the vanishing of the polynomials $\fS(g,h)$ for all $g\in \cG_{\bg}$ 
and $h\in \cG_{\bh}$. To make this process easier to follow, we set up some notation. 
In each $\fS$-polynomial $\fS(g,h)$, an item will be underlined if it can be 
written as $p\cdot b$ where $p\in S$, $b\in \cG$, and
$\ini_{\succ}(p\cdot b)\prec \lcm(\ini_{\succ}(g),\ini_{\succ}(h))$. 
We use `$\cdots$' to replace the underlined items in the
previous step.  
\begin{align*}
\fS(\bg_1,\bh_{1})&= -\underline{{x}_{kl_1}(x_{jk}+x_{l_2k})\cdot r_{ijl_2}} -
x_{jl_2}({x}_{il_1}+{x}_{jl_1})\cdot{r}_{ijk}\\
&= \cdots -({x}_{il_1}+{x}_{jl_1})\bg_1-({x}_{il_1}+{x}_{jl_1})(x_{jk}+x_{l_2k})\cdot r_{ijl_2}  \\ 
&= \cdots -\underline{({x}_{il_1}+{x}_{jl_1})\bg_1}- \underline{(x_{jk}+x_{l_2k})\bh_1^{(ijl_2)}}+
 \underline{x_{jk}x_{l_2l_1}\cdot r_{ijl_2}} +\underline{x_{l_2k}x_{l_2l_1}\cdot r_{ijl_2}} \\
\fS(\bg_1,\bh_{2})&= -x_{jk}^2\cdot r_{ijl_2}-\underline{x_{jk}x_{l_2k}\cdot r_{ijl_2}}-
x_{jl_2}{x}_{ik}\cdot{r}_{ijk} \\
&= \cdots -x_{jk}^2\cdot r_{ijl_2} -x_{ik}\bg_1-x_{ik}(x_{jk}+x_{l_2k})\cdot r_{ijl_2}  \\
&= \cdots  -\underline{x_{ik}\bg_1}-\underline{(x_{jk}+x_{l_2k})\bh_1^{(ijl_2)}}+ 
\underline{2x_{jk}x_{l_2k}\cdot r_{ijl_2}}+ \underline{x_{l_2k}^2\cdot r_{ijl_2}}  \\
\fS(\bg_1,\bh_{3})&= -\underline{x_{l_{2}k}x_{jk}\cdot r_{ijl_2}}-
 \underline{x_{l_{2}k}x_{l_2k}\cdot r_{ijl_2}}\\
\fS(\bg_1,\bh_{4})&= -\underline{x_{l_3k}(x_{jk}+x_{l_2k})\cdot r_{ijl_2}}\\
\fS(\bg_1,\bh_{5})&= -\underline{x_{l_3j}(x_{jk}+x_{l_2k})\cdot r_{ijl_2}}\\
\fS(\bg_1,\bh_{6})&= -\underline{x_{l_4k}(x_{jk}+x_{l_2k})\cdot r_{ijl_2}}\\
\fS(\bg_1,\bh_{7})&= -\underline{x_{l_4j}(x_{jk}+x_{l_2k})\cdot r_{ijl_2}}\\
\fS(\bg_1,\bh_{8})&= -\underline{x_{st}(x_{jk}+x_{l_2k})\cdot r_{ijl_2}}\\
\fS(\bg_1,\bh_{9})&= -\underline{x_{kl}x_{kp}(x_{jk}+x_{l_2k})\cdot r_{ijl_2}}\\
\fS(\bg_1,\bh_{0})&= -\underline{x_{jq}x_{kp}(x_{jk}+x_{l_2k})\cdot r_{ijl_2}} \\
\fS(\bg_2,\bh_{1})&= \underline{{x}_{kl_1}x_{jk}\cdot r_{ijl_2}}-x_{il_2}({x}_{il_1}+{x}_{jl_1})\cdot{r}_{ijk}
= \cdots - ({x}_{il_1}+{x}_{jl_1})(\bg_2-x_{jk}\cdot r_{ijl_2})\\
&= \cdots - \underline{({x}_{il_1}+{x}_{jl_1})\bg_2}+\underline{x_{jk}\bh_1^{(ijl_2)}}
-\underline{x_{jk} {x}_{l_2l_1}\cdot r_{ijl_2}}\\
\fS(\bg_2,\bh_{2})&= x_{jk}x_{jk}\cdot r_{ijl_2}-x_{il_2}{x}_{ik}\cdot{r}_{ijk} 
= x_{jk}x_{jk}\cdot r_{ijl_2}-(x_{ik}\bg_2-x_{ik}x_{jk}\cdot r_{ijl_2})  \\
&= \underline{x_{jk}\bh_1^{(ijl_2)}}-\underline{x_{jk}x_{l_2k} \cdot r_{ijl_2}} 
-\underline{x_{ik}\bg_2}  \\
\fS(\bg_2,\bh_{3})&= \underline{x_{l_{2}k}x_{jk}\cdot r_{ijl_2}}\\
\fS(\bg_2,\bh_{4})&= \underline{x_{l_3k}x_{jk}\cdot r_{ijl_2}}\\
\fS(\bg_2,\bh_{5})&= \underline{x_{l_3j}x_{jk}\cdot r_{ijl_2}}\\
\fS(\bg_2,\bh_{6})&= \underline{x_{l_4k}x_{jk}\cdot r_{ijl_2}}\\
\fS(\bg_2,\bh_{7})&= \underline{x_{l_4j}x_{jk}\cdot r_{ijl_2}}\\
\fS(\bg_2,\bh_{8})&= \underline{x_{st}x_{jk}\cdot r_{ijl_2}}\\
\fS(\bg_2,\bh_{9})&= \underline{x_{kl}x_{kp}x_{jk}\cdot r_{ijl_2}} \\
\fS(\bg_2,\bh_{0})&= \underline{x_{jq}x_{kp}x_{jk}\cdot r_{ijl_2}}\\
\fS(\bg_3,\bh_{1})&= -\underline{{x}_{kl_1}x_{jk}\cdot r_{il_3j}}-x_{il_3}({x}_{il_1}+{x}_{jl_1})\cdot{r}_{ijk}
= \cdots - ({x}_{il_1}+{x}_{jl_1})(\bg_3+x_{jk}\cdot r_{il_3j})  \\
&= \cdots - \underline{({x}_{il_1}+{x}_{jl_1})\bg_3}-\underline{x_{jk}\bh_1^{(il_3j)}}+ 
\underline{x_{jk}{x}_{l_3l_1}\cdot r_{il_3j}}   \\
\fS(\bg_3,\bh_{2})&= -\underline{x_{jk}^2\cdot r_{il_3j}}-x_{il_3}{x}_{ik}\cdot{r}_{ijk} 
= \cdots - x_{ik}(\bg_3+x_{jk}\cdot r_{il_3j})\\
&= \cdots - \underline{x_{ik}\bg_3}-\underline{x_{jk} \bh_1^{(il_3j)}}+
\underline{x_{l_3k}x_{jk}\cdot r_{il_3j}}+\underline{x_{jk}^2\cdot r_{il_3j}}\\
\fS(\bg_3,\bh_{3})&= -\underline{x_{l_{2}k}x_{jk}\cdot r_{il_3j}}\\
\fS(\bg_3,\bh_{4})&= -\underline{x_{l_3k}x_{jk}\cdot r_{il_3j}} \\
\fS(\bg_3,\bh_{5})&= -\underline{x_{l_3j}x_{jk}\cdot r_{il_3j}} \\
\fS(\bg_3,\bh_{6})&= -\underline{x_{l_4k}x_{jk}\cdot r_{il_3j}} \\
\fS(\bg_3,\bh_{7})&= -\underline{x_{l_4j}x_{jk}\cdot r_{il_3j}} \\
\fS(\bg_3,\bh_{8})&= -\underline{x_{st}x_{jk}\cdot r_{il_3j} }\\
\fS(\bg_3,\bh_{9})&= -\underline{x_{kl}x_{kp}x_{jk}\cdot r_{il_3j}}  \\
\fS(\bg_3,\bh_{0})&= -\underline{x_{jq}x_{kp}x_{jk}\cdot r_{il_3j}} \\
\fS(\bg_4,\bh_{1})&= \underline{{x}_{kl_1}x_{jk}\cdot r_{l_4ij}} 
-x_{l_4i}({x}_{il_1}+{x}_{jl_1})\cdot{r}_{ijk}\\
&= \cdots -  \underline{({x}_{il_1}}+ \underline{{x}_{jl_1})\bg_4}+
 \underline{{x}_{il_1}x_{jk}\cdot r_{l_4ij}} + \underline{{x}_{jl_1}x_{jk}\cdot r_{l_4ij}}  \\
\fS(\bg_4,\bh_{2})&= \underline{x_{jk}^2\cdot r_{l_4ij}} -x_{l_4i}{x}_{ik}\cdot{r}_{ijk} 
= \cdots - \underline{{x}_{ik}\bg_4}+ \underline{ {x}_{ik}x_{jk}\cdot r_{l_4ij}}\\
\fS(\bg_4,\bh_{3})&= \underline{x_{l_{2}k}x_{jk}\cdot r_{l_4ij}}\\
\fS(\bg_4,\bh_{4})&= \underline{x_{l_3k}x_{jk}\cdot r_{l_4ij}}\\
\fS(\bg_4,\bh_{5})&= \underline{x_{l_3j}x_{jk}\cdot r_{l_4ij}}\\
\fS(\bg_4,\bh_{6})&= x_{l_4k}x_{jk}\cdot r_{l_4ij} 
= \underline{ x_{jk}\bh_1^{(l_4ij)}}-  \underline{x_{jk}x_{ik}\cdot r_{l_4ij}}
- \underline{x_{jk}x_{jk}\cdot r_{l_4ij}} \\
\fS(\bg_4,\bh_{7})&= x_{l_4j}x_{jk}\cdot r_{l_4ij}  
= \underline{x_{jk}\bh_2^{(l_4ij)}} - \underline{x_{ij}x_{jk}\cdot r_{l_4ij}} \\
\fS(\bg_4,\bh_{8})&= x_{st}x_{jk}\cdot r_{l_4ij} 
=\left\{\begin{array}{ll}
\underline{ x_{jk}\bh_1^{(l_4ij)}}-  \underline{x_{jk}x_{it}\cdot r_{l_4ij}}
- \underline{x_{jk}x_{jt}\cdot r_{l_4ij}} , & \textrm{ for } s=l_4 \\
\underline{x_{jk}\bg_4^{(l_4ij)}}-\underline{x_{jk}x_{ij}\cdot r_{sl_4i}} & \textrm{ for } t=l_4 \\
\underline{x_{st}x_{jk}\cdot r_{l_4ij}} & \textrm{ otherwise }
\end{array}  \right.\\
\fS(\bg_4,\bh_{9})&= \underline{x_{kl}x_{kp}x_{jk}\cdot r_{l_4ij}}  \\
\fS(\bg_4,\bh_{0})&= \underline{x_{jq}x_{kp}x_{jk}\cdot r_{l_4ij}}
\end{align*}

Next, we check the vanishing of the $\fS$-polynomials of pairs of elements in $\cG_{\bg}$.
\begin{align*}
\fS(\bg_1,\bg_2)&=-x_{il_2}(x_{jk}+x_{l_2k})\cdot r_{ijl_2}-x_{jl_2}x_{jk}\cdot r_{ijl_2}   \\
&= -\underline{(x_{jk}+x_{l_2k})\bh_2^{(ijl_2)}} + \underline{x_{jl_2}x_{l_2k}\cdot r_{ijl_2}}  \\
\fS(\bg_1,\bg_3)&=-x_{il_3}(x_{jk}+x_{l_2k})\cdot r_{ijl_2}+ {x_{jl_2}x_{jk}\cdot r_{il_3j}}   \\
&=   -\underline{(x_{jk}+x_{l_2k})\bg_3^{(ijl_2)}}-\underline{x_{l_2k}x_{jl_2}\cdot r_{il_3j}} \\
\fS(\bg_1,\bg_4)&=-x_{l_4i}(x_{jk}+x_{l_2k})\cdot r_{ijl_2} - {x_{jl_2}x_{jk}\cdot r_{l_4ij}} \\
&=  -\underline{(x_{jk}+x_{l_2k})\bg_4^{(ijl_2)}}+\underline{x_{l_2k}x_{jl_2}\cdot r_{l_4ij}} \\
\fS(\bg_2,\bg_3)&= x_{il_3}x_{jk}\cdot r_{ijl_2}+x_{il_2}x_{jk}\cdot r_{il_3j} \\
&= x_{jk}\bg_3^{(ijl_2)}+x_{jk}(x_{il_2}+x_{jl_2} )\cdot r_{il_3j} \\
&= \underline{x_{jk}\bg_3^{(ijl_2)}}+\underline{x_{jk}\bh_1^{(il_3j)}}-\underline{x_{l_3l_2} x_{jk}\cdot r_{il_3j}} \\
\fS(\bg_2,\bg_4)&= x_{l_4i}x_{jk}\cdot r_{ijl_2}- {x_{il_2}x_{jk}\cdot r_{l_4ij}}  \\
&= \underline{x_{jk}\bg_4^{(ijl_2)}}- \underline{x_{jk}x_{jl_2}\cdot r_{l_4ij}}-\underline{x_{il_2}x_{jk}\cdot r_{l_4ij}}  \\
\fS(\bg_3,\bg_4)&= -x_{l_4i}x_{jk}\cdot r_{il_3j}- {x_{il_3}x_{jk}\cdot r_{l_4ij}} \\
&= - \underline{x_{jk}\bg_4^{(il_3j)}}+ \underline{x_{jk}x_{l_3j}\cdot r_{l_4il_3}}-\underline{x_{il_3}x_{jk}\cdot r_{l_4ij}} 
\end{align*}

Finally, suppose that $1\leq v_1 <k<v_2<j<v_3<i<v_4\leq n$ and  $v_{\ast}<l_{\ast}$ for $\ast=1,2,3,4$. 
We check the vanishing of the remaining $\fS$-polynomials between elements in $\cG_{\bg}$.
\begin{align*}
\fS(\bg_1,\tilde{\bg}_1)&= -x_{jv_2}x_{jk}\cdot r_{ijl_2}-\underline{x_{jv_2}x_{l_2k}\cdot r_{ijl_2}}  +
 x_{jl_2}x_{jk}\cdot r_{ijv_2}+ \underline{x_{jl_2}x_{v_2k}\cdot r_{ijv_2}} \\
 &= \cdots -x_{jv_2}x_{jk}\cdot r_{ijl_2} + 
 x_{jk}(\bg_1^{(ijv_2)}+x_{jv_2}\cdot r_{ijl_2}+ x_{l_2v_2}\cdot r_{ijl_2} ) \\
  &= \cdots + \underline{x_{jk}\bg_1^{(ijv_2)}}+ \underline{x_{jk}x_{l_2v_2}\cdot r_{ijl_2}}  \\
\fS(\bg_2,\tilde{\bg}_2)&= x_{iv_2}x_{jk}\cdot r_{ijl_2}  - x_{il_2}x_{jk}\cdot r_{ijv_2} \\
  &= x_{jk}(x_{iv_2}+x_{jv_2})\cdot r_{ijl_2}  - x_{jk} \bg_2^{(ijv_2)} \\
     &=\underline{x_{jk}\bh_1^{(ijl_2)}}-\underline{x_{jk}x_{l_2v_2}\cdot r_{ijl_2}}  
     - \underline{x_{jk} \bg_2^{(ijv_2)}} \\
\fS(\bg_3,\tilde{\bg}_3)&= -x_{iv_3}x_{jk}\cdot r_{il_3j} +x_{il_3}x_{jk}\cdot r_{iv_3j} \\
&=-x_{jk}(\bg_2^{(il_3j)}-x_{l_3j} \cdot r_{il_3v_3}) +x_{jk}(\bg_3^{(iv_3j)}+x_{v_3j}\cdot r_{il_3v_3}) \\
&=-\underline{x_{jk}\bg_2^{(il_3j)}}+\underline{x_{jk}x_{l_3j} \cdot r_{il_3v_3}} +
\underline{x_{jk}\bg_3^{(iv_3j)}}+\underline{x_{jk}x_{v_3j}\cdot r_{il_3v_3}} \\
\fS(\bg_4,\tilde{\bg}_4)&= \underline{x_{v_4i}x_{jk}\cdot r_{l_4ij}} -
\underline{x_{l_4i}x_{jk}\cdot r_{v_4ij}}
\end{align*}

Therefore, all the $\fS$-polynomials from $\cG$ vanish, and so $\cG$ is a Gr\"{o}bner 
basis for $\im(\Psi)$.  This completes the proof of Theorem \ref{thm:GrobnerBasis}.
\hfill \qed

\begin{ack}
Computations using Macaulay~2 \cite{Macaulay2} were essential to gaining 
the intuition that led to the results of this paper.  We thank Paolo Bellingeri 
for useful discussions on the McCool groups. We also thank the referee for a 
careful reading of the manuscript and for pertinent remarks. 
\end{ack}

\bibliographystyle{amsplain}

\newcommand{\doi}[1]
{\texttt{\href{http://dx.doi.org/#1}{doi:#1}}}

\newcommand{\arxiv}[1]
{\texttt{\href{http://arxiv.org/abs/#1}{arXiv:#1}}}
\newcommand{\arxi}[1]
{\texttt{\href{http://arxiv.org/abs/#1}{arxiv:}}
\texttt{\href{http://arxiv.org/abs/#1}{#1}}}
\renewcommand{\MR}[1]
{\href{http://www.ams.org/mathscinet-getitem?mr=#1}{MR#1}}

\end{document}